\documentclass{article}

\usepackage{algorithm}[0.1]

\usepackage{arxiv}

\usepackage[utf8]{inputenc} 
\usepackage[T1]{fontenc}    
\usepackage{hyperref}       
\usepackage{url}            
\usepackage{booktabs}       
\usepackage{amsfonts}       
\usepackage{nicefrac}       
\usepackage{microtype}      
\usepackage{lipsum}
\usepackage{graphicx}
\graphicspath{ {./images/} }
\usepackage{cite}
\usepackage{amsmath, amssymb, amsfonts, caption}
\usepackage{algorithmic}
\usepackage{graphicx}
\usepackage{algorithm,algorithmic}
\usepackage{hyperref}
\usepackage{textcomp}
\newtheorem{theorem}{Theorem}[section]
\newtheorem{problem}[theorem]{Problem}
\newtheorem{definition}[theorem]{Definition}
\newtheorem{assumption}[theorem]{Assumption}
\newtheorem{proposition}[theorem]{Proposition}

\newtheorem{lemma}[theorem]{Lemma}
\newtheorem{example}[theorem]{Example}
\newtheorem{remark}[theorem]{Remark}
\newtheorem{procedure}{Procedure}
\newtheorem{proof}{Proof}
\usepackage{tikz,xcolor,pgfplots}
\usepackage{url}
\usetikzlibrary{shapes,arrows,calc,positioning,arrows.meta,graphs}
\usepackage{lipsum}
\usepackage{amsfonts,graphicx,epstopdf,algorithmic}
\usepackage[outline]{contour} 
\contourlength{1.2pt}
\usetikzlibrary{positioning,calc}
\usetikzlibrary{backgrounds}

\ifpdf
  \DeclareGraphicsExtensions{.eps,.pdf,.png,.jpg}
\else

  \DeclareGraphicsExtensions{.eps}
\fi

\pgfmathdeclarefunction{gauss}{3}{%
  \pgfmathparse{0.25*1/(#3*sqrt(2*pi))*exp(-((#1-#2)^2)/(2*#3^2))}%
}
\pgfmathdeclarefunction{cdf}{3}{%
  \pgfmathparse{0.3*1/(1+exp(-0.07056*((#1-#2)/#3)^3 - 1.5976*(#1-#2)/#3))}%
}
\pgfmathdeclarefunction{fq}{3}{%
  \pgfmathparse{0.3*1/(sqrt(2*pi*#1))*exp(-(sqrt(#1)-#2/#3)^2/2)}%
}
\pgfmathdeclarefunction{fq0}{1}{%
  \pgfmathparse{0.3*1/(sqrt(2*pi*#1))*exp(-#1/2))}%
}
\pgfmathdeclarefunction{fq1}{1}{%
  \pgfmathparse{0.05*(8*(#1-16)*(#1-12)*(#1-8)*(#1-4)/((-16) * (-12) *(-8) * (-4))+ 1*(#1-16)*(#1-12)*(#1-8)*#1/((-12) * (-8) * (-4) * 4)+ 6*(#1-16)*(#1-12)*(#1-4)*#1/ ((-8) * (-4)* 4 * 8)+ 3*(#1-16)*(#1-8)*(#1-4)*#1/ ((-4)* 4 * 8 * 12)+ 8*(#1-12)*(#1-8)*(#1-4)*#1/ (4 * 8 * 12 * 16))}
  }

\colorlet{mydarkblue}{blue!30!black}

\usepgfplotslibrary{fillbetween}
\usetikzlibrary{patterns}
\pgfplotsset{compat=1.12} 

\usepackage{enumitem}
\setlist[enumerate]{leftmargin=.5in}
\setlist[itemize]{leftmargin=.5in}

\newcommand{\convexhull}{\mbox{$\mathrm{convh}$}}

\newcommand{\sign}{\mbox{$\mathrm{sign}$}}
\newcommand{\interior}{\mbox{$\mathrm{intrior}$}}
\newcommand{\Range}{\mbox{$\mathrm{R}$}}
\usepackage{amsopn}
\DeclareMathOperator{\diag}{diag}

\begin{document}

\title{
\vspace{8mm}
Optimal Control of a Stochastic Power System \\ 
-- Algorithms and Mathematical Analysis 
}
\author{Zhen Wang
\thanks{
Zhen Wang is with the School of Mathematics, Shandong University, 
Jinan, 250100, Shandong Province, China and
is visiting the Dept. of Applied Mathematics, Delft University of Technology, 
Delft, The Netherlands from November 2021 till February 2024 (email: wangzhen17@mail.sdu.edu.cn). Zhen Wang is grateful to the China Scholarship Council (No. 202106220104)
for financial support for his stay at Delft University of Technology
in Delft, The Netherlands.
}~,\\
Kaihua Xi \footnotemark[2]~,
Aijie Cheng 
\thanks{
Kaihua Xi and Aijie Cheng are with 
the School of Mathematics, Shandong University, Jinan, 
250100, Shandong Province, China 
(email: kxi@sdu.edu.cn, aijie@sdu.edu.cn).
}~, \\
Hai Xiang Lin \footnotemark[3]~,  
Jan H. van Schuppen
\thanks{
Hai Xiang Lin and Jan H. van Schuppen are with the
Dept. of Applied Mathematics, Delft University of Technology, 
Delft, The Netherlands
(email: H.X.Lin@tudelft.nl, 
j.h.vanschuppen@tudelft.nl, 
vanschuppenjanh@freedom.nl).
}~.
\vspace{-3mm}
}

\maketitle
\begin{abstract}
The considered optimal control problem of a stochastic power system,
is to select the set of power supply vectors
which infimizes the probability
that the phase-angle differences of any power flow of the network,
endangers the transient stability of the power system 
by leaving a critical subset.
The set of control laws is restricted to be a periodically
recomputed set of fixed power supply vectors based
on predictions of power demand for the next short horizon.
Neither state feedback nor output feedback is used.
The associated control objective function
is Lipschitz continuous, nondifferentiable, and nonconvex.
The results of the paper include
that a minimum exists in the value range of the control objective function.
Furthermore, it includes a two-step procedure to compute an approximate minimizer 
based on two key methods: 
(1) a projected generalized subgradient method 
for computing an initial vector, and 
(2) a steepest descent method for approximating a local minimizer. 
Finally, it includes two convergence theorems 
that an approximation sequence converges to a local minimum. 
\end{abstract}
\par\vspace{1\baselineskip}\par\noindent

\section{Introduction}\label{sec:123}
\paragraph{Motivation}
The primary motivation for the optimal control problem investigated in this paper is
to safeguard transient stability of a stochastic power system. 
The control objective is 
to find that power supply vector in a domain
which minimizes the control objective function. 
The control objective function quantifies the performance 
of the stochastic power system for each power supply vector. 
The value of the control objective function
is a threshold for the probability that 
the power flow in any power line exits a critical subset.
The readers are referred to a companion paper
\cite{zhenwang:reportone:2023} for further details.
\par
The {\em control objective function} is formulated 
as the maximum value among a finite set of secondary functions. 
Each secondary function is a sum of two constituent functions: 
one component denoted as $\sigma_k$, 
exhibits non-convex characteristics, 
while the other component referred to as $m_k$, 
may lack differentiability in cases where a component 
of the argument becomes zero.
The control objective function satisfies
a generic differentiability property 
but exhibits nondifferentiability 
on either an algebraic subset or a subset of lower algebraic complexity. 
The domain of power supply vectors denoted as $P^+$, 
takes the form of a polytope.
\paragraph{The Optimal Control Problem}
The approach to the optimal control problem,
based on periodic recomputation of an optimal set of power supply vectors,
leads to an infimization problem for a control objective function.
That function is Lipschitz continuous, nondifferentiable, and nonconvex.
The nondifferentiability stems from the representation 
of the function as the maximum of several functions, 
each such function containing a term that defies differentiability at zero. 
The nonconvexity is due to the Lyapunov matrix equation.

%
\paragraph{Literature review} 
Analysis and control of a stochastic power system
have been treated in a number of papers,
\cite{anderson:bose:1983,demarco:bergen:1987,dong:zhao:hill:2012}.
Investigated have been simulations
and the approximation of 
the probability distribution of the first exit time
of the domain of attraction.
The latter approach requires 
computing the type-one synchronous states and
numerically approximating the probability distribution.
The latter approach is 
computationally difficult with a high computational complexity.
To the best of the author's knowledge,
they have not found the approach of the optimization problem 
in the literature.
The optimal control problem treated in this paper is transcribed to a constrained nonlinear optimization problem
which happens to be nonconvex and nondifferentiable.
It turns out to be a challenging task to optimize nonconvex and nondifferentiable objective functions, 
while they are frequently encountered in control engineering when random variables have been introduced.
\par 
A nonconvex optimization problem is difficult because the objective function has many minimizers. A nondifferentiable optimization problem is even more challenging 
because the gradient which is a fundamental tool that is often used in the iterations of optimization algorithms, is not everywhere defined. 
Instead, either a generalized subgradient or a generalized gradient is used 
to generalize the concept of a gradient 
to the set of nondifferentiable functions. 
The initial research of F. Clarke on subgradients 
\cite{clarke1975generalized,clarke2013functional} is well known.
B. Mordukhovich and others developed optimization of nondifferentiable and nonconvex functions,
\cite{mordukhovich:2006:volumeone,mordukhovich:2006:volumetwo,cui2021modern,mordukhovich2023easy}.
For basic insights into semiconvexity, semicontinuity, and various properties associated with these function classes,
we refer to \cite{mifflin1977semismooth}. 
Additionally, more comprehensive information on generalized gradients and 
generalized subgradients can be found in references such as 
\cite{clarke1975generalized,rockafellar1980generalized,balder:2009}.
Readers who are interested in optimization criteria and 
methodologies related to these optimization techniques, 
we recommend exploring in 
\cite{canovas2010variational,mordukhovich2013subdifferentials}.
\par
In addition to the generalized gradient, which represents a vector class, researchers have also leveraged directional derivatives to optimize nonsmooth functions. Notably, A. Ben-tal introduced second-order necessary and sufficient optimality criteria for four distinct types of nonsmooth minimization problems in \cite{ben1982necessary}. Subsequently, the same author extended the application of both first and second directional derivatives to establish necessary and sufficient conditions for a strict local minimizer, as elaborated in \cite{ben1985directional}. Expanding on this work, A. Shapiro provided a comprehensive overview of various directional derivative concepts in \cite{shapiro1990concepts}.
\par
Many algorithms have been developed for nonconvex and nonsmooth optimization. T. Liu et al. introduced the successive difference-of-convex approximation method, demonstrating its global convergence for a specific class of objective functions in \cite{liu2019successive}. In a similar vein, Y. Wang proposed the ADMM algorithm for nonconvex nonsmooth optimization, which also exhibits global convergence, particularly for objective functions with equality constraints, as discussed in \cite{wang2019global}.
Another noteworthy contribution comes from B. Wen et al., who presented the proximal gradient algorithm with extrapolation and conducted a comprehensive convergence analysis, referred to as R-linear convergence in \cite{wen2017linear}. 
\par 
In this paper, both the generalized subgradient and the directional derivative will be used to infimize our control objective function. We will use a two-step method, the first step uses a projected generalized subgradient method to determine an initial position that is in the domain of attraction of a relatively good local minimizer, then the second step presents a steepest descent method starting from the initial position to converge to the local minimizer. Because the iteration is based on either the generalized subgradient or the first directional derivative, the proposed algorithms are first-order algorithms.
\paragraph{Contributions of this Paper} 
(1) The continuity and the differentials of the control objective function 
per direction vector.
(2) The existence of a minimizer. 
(3) An algorithm to compute the gradient vector and 
the Hessian matrix of the standard deviation that served as an implicit function within the control objective function, and
two algorithms, 
(3.A) to compute a stationary vector as an initial vector 
for the next algorithm, and 
(3.B) to approximate a local minimizer. 
(4) Theorems which imply that an approximation sequence,
produced by the algorithms,
converges to a value or to a local minimum.
(5) An example which demonstates the proposed algorithms.
\paragraph{Paper Organization}
Section \ref{sec:problemformulation} 
provides an introduction to power system fundamentals, 
the domain of power supply vectors, and
the procedure for the calculation of the control objective function. 
In Section \ref{sec:minimizerexistence}, 
we establish the existence of a minimizer and 
analyze the control objective function's continuity. 
Section \ref{sec:convexityanddifferential} examines convexity, 
and provides the first and the second directional derivatives.
Section \ref{sec:algorithms}  presents three algorithms: 
(1) for computing the gradient vector and the Hessian matrix 
of a standard deviation $\sigma_i,\:i\in\mathbb{Z}_{n_E}$, 
(2) an algorithm based on the steepest descent method, and 
(3) an algorithm using the projected generalized subgradient method;
along with convergence theorems. 
Section \ref{sec:examples} uses an academic example to demonstrate the 
proposed algorithms.
Section \ref{conclusionfurtherinvestigation}
summarizes conclusions and outlines open research challenges.
\section{From an optimal control problem to a nonlinear optimization problem}\label{sec:problemformulation}
\paragraph{Notation} 
The following mathematical notation is used in this paper. 
Denote the integers, the positive integers, and the natural numbers 
respectively by
$\mathbb{Z}$, $\mathbb{Z}_+ = \{ 1, ~ 2, ~ \ldots \}$, and 
$\mathbb{N} = \{ 0, ~ 1, ~ 2, ~ \ldots \}$.
For any integer $n \in \mathbb{Z}_+$ denote the finite sets
$\mathbb{Z}_{n}=\{1,2,\cdots,n\}$ and $\mathbb{N}_{n}=\{0,1,2,\cdots,n\}$. 
Denote the positive and the strictly positive real numbers respectively by
$\mathbb{R}_{+}$ and $\mathbb{R}_{s,+}$.
Define the sign function as,
$\sign(x)=+1$ if $x>0$, $-1$ if $x<0$, and $0$ if $x=0$.
The $n$-dimensional Euclidean space is denoted by
$\mathbb{R}^n$ 
and it is equiped with the inner product $\langle\cdot,\cdot\rangle$, 
the infinity norm $\|\cdot\|_\infty$, and 
the Euclidean norm $\|\cdot\|_2$. For any integer $n \in \mathbb{Z}_+$,
$\mathbb{R}_+^n$ and $\mathbb{R}_{s,+}^n$ denote respectively
the $n$-fold product of $\mathbb{R}_+$ and $\mathbb{R}_{s,+}$.
\par
The set of matrices of size $m \times n$ 
with elements of the real numbers is denoted by
$\mathbb{R}^{m \times n}$.
The matrix transpose operator is denoted by $\top$.
The spectrum norm and the Frobenius norm of a matrix
are denoted by $\|\cdot\|_{2,s}$ and $\|\cdot\|_F$, respectively.
Denote the n-th row of matrix $A\in \mathbb{R}^{m \times n}$ by $A(n)$ and 
the n-th column of matrix $A$ is denoted by $A_n$.
Denote the set of diagonal matrices of size $n \times n$ 
with positive or strictly positive elements as 
$\mathbb{R}_{+,\diag}^{n \times n}$ and 
$\mathbb{R}_{s+,\diag}^{n \times n}$, respectively.
A diagonal matrix 
with elements from the vector $v\in \mathbb{R}^n$ 
is represented as $\diag(v)\in \mathbb{R}^{n \times n}$.
Unit vectors: We use $e_k$ to represent the k-th unit vector. 
Identity matrix: We denote the identity matrix of size $n$ as $I_{n}$.
Column selection: For $m>n$, we denote the n-th to m-th columns of the product matrices $\left(A\: B\: C\right)$ as $(A\: B\: C)([n:m])$.
\par
Hadamard product and power: The Hadamard product and Hadamard power, representing component-wise multiplication or power, is denoted as $A\circ B$ or $A\circ^x$, separately. 
The scalar-vector product of a real number $c \in \mathbb{R}$ 
and a vector $v \in \mathbb{R}^n$ is denoted by $c \cdot v$.
\subsection{Introduction to the optimal control problem}
Control of a power system is treated
with a focus on frequency control;
voltage control is not considered as is customary in frequency control.
The approach of secondary frequency control is used
in which there is a sequence of short horizons of a few minutes, 
think of 3 to 5 minutes.
In each short horizon,
the power supply is adjusted and kept constant during this short horizon.
\par
The control objective is to safeguard transient stability.
A generally accepted sufficient condition for stability
is that the phase-angle difference over any power line
remains in the critical subset  $(- \pi/2, ~ +\pi/2) \subset \mathbb{R}$.
Currently, there is no satisfactory proof that this condition
is sufficient in the literature.
\par
Due to disturbances of power sources and of power loads,
which are expected to increase in the future,
the phase-angle differences become random variables.
The performance criterion to minimize is thus
the probability that the phase-angle differences 
over all power lines to exit the critical subset is less than a set threshold $\epsilon \in (0, ~ + 1)$.
The reader may think of the probability to be $10^{-3}$.
Then one obtains the control objective function 
defined below, in which other minor approximations are used.
For further details, the reader is referred to the companion paper
\cite{zhenwang:reportone:2023}.
Since state feedback and output feedback are not considered in this paper, we
need to solve a constrained nonlinear optimization problem.
\subsection{The domain of power supply vectors}\label{sec:domain}
\paragraph{Power System Preliminaries}
The power network is modelled by a graph $G = (V, ~ E)$ 
with the set of vertices $V$ and the set of edges in $E$.
There are $n_V \in \mathbb{Z}_+$ vertices
and $n_E \in \mathbb{Z}_+$ edges.
A line between two busses of a power network
is modelled by an edge denoted by
$k = (i_k, ~ j_k) \in E$ which connects vertices $i_k$ and $j_k$.
Nodes with providing power supply are indexed according to
$n_1, ~ n_2,~ \ldots, ~ n^+$ and
nodes with only power demand are indexed by
$n^++1,n^++2,\ldots,n_V$. 
The diagonal matrices of the strictly positive inertias
and the positive damping coefficients are denoted by
$M = \diag([m_1,\ldots,m_{n_V}])$ and 
$D = \diag([d_1,\ldots,d_{n_V}])$.
The network incidence matrix is 
$B\in\mathbb{R}^{n_V\times n_E}$
and the weight matrix is 
$W =\diag([w_1,\ldots,w_{n_E}])\in\mathbb{R}_{s+,\diag}^{n_E\times n_E}$.
Denote by the matrix $K_2\in \mathbb{R}_{+,\diag}^{n_{V}\times n_{V}}$ 
the standard deviation of the 
vector-valued Brownian motion acting on the frequencies of the nodes.
\par
Denote, for a short horizon,
the maximal available power supply by 
$p^{+,max} \in \mathbb{R}_+^{n^+}$;
the prediction of the power demand by
$p^{-} \in \mathbb{R}_+^{n_V-n^+}$;
the sum of the maximal power supply by
$p_{sum}^{+,max} =\sum_{i=1}^{n^+} ~ p_i^{+,max}\in \mathbb{R}_{s,+}$;
and the sum of the power demand by
$p_{sum}^- =\sum_{i=n^++1}^{n_{V}} ~ p_i^{-}\in \mathbb{R}_{s,+}$.
It is assumed that the sum of the available power supply
is larger than or equal to the sum of the predicted demand,
$p_{sum}^{+,max} \geq p_{sum}^{-}$.
Because the sum of the power supply has to equal
the sum of the predicted power demand,
the decision vector of the power supply can be defined as 
$p_s= \left( p_1^+, ~ \ldots, ~ p_{n^+-1}^+ \right)^\top \in \mathbb{R}_+^{n^+-1}$.
From $p_s$ one can compute the last element of the vector of power supplies,
$p_{n^+}^+ = p_{sum}^- -  \sum_{i=1}^{n^+-1} p_s\left(i\right)$.
\begin{definition}\label{dom:DODEFINITION2}
Define the domain of the power supply vector $p_s\in\mathbb{R}_+^{n^+-1}$,
depending on the maximal power supply $p^{+,max}$ and 
the sum of the power demands $p_{sum}^-$,
as the set, 
\vspace{-2.8mm}
\begin{align}\label{DODEFINITION2}
    P^+
	& = P^+( p^{+,max}, ~ p_{sum}^- ) =
	\left\{
          p_s \in \mathbb{R}_+^{n^+-1} |~ 
	  b_1 \leq A_1\:  p_s, ~
	  p_s \leq b_2
        \right\},\\
                                      b_1\left(i\right)
    & = p_{sum}^--p_{n^+}^{+,max}-\cdots-p_{i+1}^{+,max}, ~
                                                                 b_2 \left(i\right) 
     = p_{i}^{+,max},~
      b_1,\:b_2\in \mathbb{R}^{n^+-1},\\
          A_1\left(i,j\right)
    & =
      \begin{cases}
        1 \quad \text{if} ~ i \leq j\\
        0 \quad \text{if} ~ i> j
      \end{cases}, ~
      A_1\in \mathbb{R}^{(n^+-1)\times (n^+-1)}.
         \end{align}
The domain $P^+$ is set to meet the power demand and realistic constraints.
\end{definition}
\begin{proposition}\label{thmdomain}
The domain $P^+$ is compact and convex. Moreover, it is a polytope.
\end{proposition}
\subsection{The control objective function}
\begin{definition}\label{def:controlobjectivefunction}
Define the control objective and related functions according to,
\begin{align*}
    & ~~~~  \forall ~ k \in \mathbb{Z}_{n_E}, ~ 
        k = \left(i_k, \:j_k\right), ~  
        f_{as}: P^+ \rightarrow \mathbb{R}^{n_E},~
        \forall ~ p_s \in P^+, \nonumber \\
      m_{k}\left(p_s\right) 
    & = 
         f_{as,k}\left(p_s\right) = \arcsin\left( |\left (A \: p_s + b\right)_k | \right),~ f_{as,k}: P^+ \rightarrow \mathbb{R}_+,\\
             f_k\left(p_s\right)
    & =  f_{as,k}\left(p_s\right) + r \cdot \sigma_k\left(p_s\right) \in \mathbb{R}_+, ~
            \forall ~ k \in \mathbb{Z}_{n_E},~ f_{k}: P^+ \rightarrow \mathbb{R}_+, \nonumber  \\
        f\left(p_s\right)
    & =  \| f_k\left(p_s\right) \|_{\infty} 
          = \max_{k \in \mathbb{Z}_{n_E}} ~ f_k\left(p_s\right), ~  f: P^+ \rightarrow \mathbb{R}_+.\nonumber
       \end{align*} %
\end{definition}
The control objective function $f$ 
represents the maximum probability threshold across all power lines 
associated with the parameter $r\in (0,\infty)$ 
according to the invariant probability distribution. 
The function $f_{as,k}= m_k$ with respect to the power supply vector $p_s$
stands for the absolute value of 
the mean of the phase-angle difference of power line $k$. 
See \cite[pp. 95--97]{bourbaki:2004}
for properties of the $\arcsin$ function.
The function $\sigma_k$ 
represents the standard deviation of the phase-angle difference of that line.
Since it is an implicit function within the objective function $f$, it will be 
also referred to as an implicit function in the following sections.
\begin{procedure}\label{proc:computationcontrolobjectivefunction}
{\em Computation of the Control Objective Function for a Power Supply Vector}.
\par
Input data.
The parameters of the power system including 
the graph of the connected power network $G = (V, ~ E)$ and
the matrices $M, ~ D, ~ B, ~ W, ~ K_2$.
The vectors $p^{+,max}$ and $p^-$.
A power supply vector $p_s \in P^+$ and the vector of power demands
$p^-
= \begin{bmatrix}
    p_{n^++1}^{-} & \cdots &p_{n}^{-}
  \end{bmatrix}^\top \in \mathbb{R}^{n_V-n^+}$.
Finally the parameter of the control objective function $r \in \mathbb{R}_{s+}$.
\begin{enumerate}
\item
Compute an orthonormal matrix $U$, the matrix $A$, and the vector $b$ according to,
\begin{align*}
	 U^\top \: \Lambda^{\dagger}\:U
    & = \left(
	  B\:W\:B^\top
	\right)^{\dagger}, ~~
        U
        = \begin{bmatrix}
	    U_1 & U_2 & \cdots & U_{n_E}
        \end{bmatrix}, \\
    A
    & = B^{\top}\: U^{\top}\: \Lambda^{\dagger}\: 
	\begin{bmatrix}U_1-U_{n^+},&U_2-U_{n^+},&\cdots&,U_{n^+-1}-U_{n^+}
	\end{bmatrix}\in\mathbb{R}^{n_E\times (n^+-1)},\\
    b
    & = B^{\top}\: U^{\top}\: \Lambda^{\dagger}\: 
	\begin{bmatrix}
		U_{n^++1}-U_{n^+},&\cdots&,U_{n}-U_{n^+}
	\end{bmatrix}(-p^-)\in\mathbb{R}^{n_E}.
\end{align*}
\item
Solve the following Lyapunov matrix equation
to obtain the matrix $Q_y \in \mathbb{R}^{n_E \times n_E}$,
\begin{align}
    0
  & = J \left( 
	  p_s 
	\right) \:Q_x + Q_x\:J
	\left(
	 p_s
	\right)^{\top}
	+ K\:K^{\top}, ~ \label{Lyapeqn}\\
    Q_y
    & = C\:Q_x\:C^{\top},~ \mbox{where}, \label{Lyapouteqn} \\
    J\left(p_s\right)
  & = \begin{bmatrix}
            0_{n_{V} \times n_{V}}   & I_{n_V} \\
            - M^{-1} \:B \:W\: F(p_s) & - M^{-1} \:D 
      \end{bmatrix} \in \mathbb{R}^{2 n_V \times 2 n_V}, \\ 
        F\left(p_s\right)
   & = \diag
       \left(
	 \cos
	 \left(
	   \arcsin\left(A\:p_s + b\right)
	 \right)
       \right)\: B^\top
       \in \mathbb{R}^{n_E \times n_V}, \\
     K
   & = \begin{bmatrix}
              0 \\ K_2
       \end{bmatrix} \in ~\mathbb{R}^{2n_{V}\times n_{V}},~~
     C = \begin{bmatrix}
              B^\top & 0
         \end{bmatrix}\in ~\mathbb{R}^{n_{E}\times 2n_{V}}.
\end{align}
\item
Compute the values of standard deviation $\sigma_k$ and variance $V_k$ for the vector $p_s \in P^+$,
\begin{align*}
    \sigma_k\left(p_s\right) 
    & = Q_y\left(k,k\right)^{\frac{1}{2}}, ~ 
      \sigma \left(p_s\right)
      = \begin{bmatrix}
	    \sigma_1\left(p_s\right) & \cdots & \sigma_{n_E} \left(p_s\right)
        \end{bmatrix}^\top, ~
	\sigma: P^+ \rightarrow \mathbb{R}^{n_E}, \\
    V_k\left(p_s\right) 
    & = \sigma_k^2 \left( p_s \right), ~ 
      V\left(p_s\right)
      =
      \begin{bmatrix}
        V_1\left(p_s\right), &\cdots, &V_{n_E}\left(p_s\right)
      \end{bmatrix}^\top, ~
      V: P^+\to \mathbb{R}^{n_E}.
\end{align*}
\item
Compute the value of the control objective function
for the vector $p_s \in P^+$,
\begin{align*}
    m_{k}\left(p_s\right) 
    & = f_{as,k}\left(p_s\right) 
	= \arcsin\left( |\left (A \: p_s + b\right)_k | \right)=\arcsin\left( |A\left(k\right)\: p_s + b_k| \right),~ \\
    f_k\left(p_s\right)
    & = f_{as,k} \left(p_s\right) 
	+ r \cdot \sigma_k\left(p_s\right) \in \mathbb{R}_+, ~
	\forall ~ k = (i_k, ~ j_k ) \in \mathbb{Z}_{n_E}, \\
    f\left(p_s\right)
    & = \| f_k\left(p_s\right) \|_{\infty} 
          = \max_{k \in \mathbb{Z}_{n_E}} ~ f_k\left(p_s\right). 
\end{align*}
\end{enumerate}
\end{procedure}
\begin{remark}
If the matrix $J\left(p_s\right)$ is Hurwitz, 
then the matrix $Q_x$ is strictly positive definite. 
But in fact, the system matrix $J(\theta_s,0)$ is not Hurwitz 
because it has a zero eigenvalue due to 
the product matrix $BWF(p_s)$ having a zero eigenvalue. 
Therefore, there is a need for a reduction procedure, 
see \cite{WANG2023110884,wang2023siam}.  
Also Lemma \ref{Vaindiff} will show how to do the reduction.
\end{remark}
\subsection{The expression of constrained nonlinear optimization problem}\label{sec:problem}
Control of a stochastic power system leads
to the following constrained nonlinear optimization problem.
\begin{problem}\label{probleminfimizefunctionf}
Solve the infimization problem for the control objective function,
and determine a value $a \in \mathbb{R}$ and a minimizer
$p_s^*\in P^+$ such that,
\begin{align*}
 a
    & =  f\left(p_s^*\right)= \inf_{p_s \in P^+} ~ f\left(p_s\right)
     = \inf_{p_s \in P^+} ~\max_{k\in\mathbb{Z}_{n_{E}}}{
    \left[m_{k}\left(p_s\right)+r\cdot\sigma_k\left(p_s\right)\right]}.
\end{align*}
\end{problem}
\section{The existence of a minimizer of the control objective function}\label{sec:minimizerexistence}
The existence of a minimizer of a nonconvex function is investigated in \cite{ekeland1979nonconvex}.
\begin{theorem}\label{existence}
There exists a minimizer $p_s^*\in P^+$ such that $f\left(p_s^*\right)=\inf_{p_s\in P^+}{f\left(p_s\right)}$. 
It is important to note that this minimizer is generally not unique.
\end{theorem}
\begin{proof}
The proof of this theorem is based on the following steps, each of which is demonstrated in the subsequent subsections:
\begin{enumerate}
\item Prove that the control objective function is Lipschitz continuous.
\item Prove that if the domain of a Lipschitz continuous function is both compact and convex, 
then the value range of this function is also compact and convex.
\item 
Prove that the value range of the control objective function being a convex and compact subset 
of $\mathbb{R}$
implies that the range is a closed interval, denoted as 
$R\left(f\right)=\left[a,~c\right]$ for $a,~c \in \mathbb{R}$ where $-\infty<a\leq c$.
\item 
Prove the existence of a minimizer $p_s^* \in P^+$
such that $f\left(p_s^*\right) = \inf_{p_s \in P^+}f\left(p_s\right) = a$. 
\end{enumerate}
\end{proof}
\subsection{The continuity of the control objective function}\label{LipschitzContinuity}
This subsection provides the proof of the Step $1$ of Theorem \ref{existence}. The subsection is organized as follows:
(1) Analysing the continuity of the implicit function $\sigma$ on $P^+$. (2) Demonstrating the continuity of the function $f_{as}$ on $P^+$. (3) Proving the continuity of the control objective function $f$ on $P^+$.
\begin{lemma}\label{Vaindiff}
The standard deviation of the phase-angle differences, 
$\sigma: P^+ \to \mathbb{R}_{s,+}^{n_{E}}$, 
is an infinitely differentiable function on $P^+$, denoted as $\sigma\in C^{\infty}(P^+)$.
\end{lemma}
\begin{proof}
To prove this result, we proceed with the following steps:
(1) After linearization of the power system at a synchronous state, define the weight matrix function $W$ with respect to $p_s\in P^+$, as $W\left(p_s\right)=W \: \diag\left(\cos\left(\arcsin\left(A \: p_s + b\right)\right)\right)$. This leads to the equation $B \: W \: F(\theta_s) = B \: W(p_s) \: B^\top$.
(2) Transform the matrix product $M^{-1} \: B \:W(p_s) \:B^\top$ into a diagonal matrix whose first element is equal to zero using the matrix $M^{-\frac{1}{2}} \:U(p_s)$. Here, the matrix $U(p_s)\in\mathbb{R}^{n_{V}\times n_{V}}$ is an orthonormal matrix. This transformation is applied to both matrices $K$ and $C$, simultaneously.
\begin{equation}\label{Systemmatrix2}
\begin{aligned}
J_c\left(p_s\right)&=\left[
 \begin{array}{ll}
 0_{n_{V} \times n_{V}},&I_{n_{V}}\\
 -U\left(p_s\right)^\top M^{-\frac{1}{2}}B\:W\left(p_s\right)\:B^\top M^{-\frac{1}{2}}U\left(p_s\right),&-U\left(p_s\right)^\top M^{-1}\:D\:U\left(p_s\right)
 \end{array}
 \right]\\
 &\in \mathbb{R}^{2n_{V}\times 2 n_{V}},\\
 ~ K_c\left(p_s\right)&=\left[
 \begin{array}{ll}
 0_{n_{V} \times n_{V}}\\U\left(p_s\right)^\top M^{-\frac{1}{2}}\: K_2
 \end{array}
 \right]\in \mathbb{R}^{2n_{V}\times n_{V}},
  C_c\left(p_s\right)=\left[
  \begin{array}{ll}B^\top\:M^{-\frac{1}{2}}\: U\left(p_s\right),&0
  \end{array}
  \right]\in \mathbb{R}^{n_{E}\times 2 n_{V}}.
  \end{aligned}
  \end{equation}
  \par These matrices are modified as follows to eliminate eigenvalue 0 of $J_c$, which is the reason that equation (\ref{Lyapeqn}) is unsolvable. 
  The first row and first column of $J_c$ are removed. The first row of $K_c$ is removed. The first column of $C_c$ is removed. After that, we can get,
 \begin{equation}\label{Systemmatrix3}
 \begin{aligned}
 J_d\left(p_s\right)&=\left[\begin{array}{ll}0_{\left(n_{V}-1)\times (n_{V}-1\right)},&0_{\left(n_{V}-1\right)\times 1}~|~I_{n_{V}-1}\\ -\left(U\left(p_s\right)^\top M^{-\frac{1}{2}}\:J\: M^{-\frac{1}{2}}U\left(p_s\right)\right)\left([2:n_{V}]\right),&-U(p_s)^\top M^{-1}\:D\:U(p_s)\end{array}\right],\\
& \in \mathbb{R}^{(2n_{V}-1)\times (2n_{V}-1)}\\
 K_d\left(p_s\right)&=\left[\begin{array}{ll}0_{\left(n_{V}-1\right)\times n_{V}}\\U\left(p_s\right)^\top M^{-\frac{1}{2}}\:K_2\end{array}\right],
 C_d\left(p_s\right)=\left[\begin{array}{ll}\left(B^\top M^{-\frac{1}{2}}U(p_s)\right)\left([2:n_{V}]\right),&0_{n_{E} \times n_{V}}\end{array}\right],\\
& \in \mathbb{R}^{\left(2n_{V}-1\right)\times n_{V}},\qquad\qquad\in \mathbb{R}^{n_{E}\times (2n_{V}-1)}
 \end{aligned}
 \end{equation}
 It's worth noting that, following the deduction from Equation (\ref{Systemmatrix2}) to Equation (\ref{Systemmatrix3}), the matrix $J_d\left(p_s\right)$ becomes Hurwitz. Therefore, Equation (\ref{Lyapeqn}) possesses a positive-definite solution. 
Additionally, the value of the implicit function $\sigma_k$ for a vector $p_s$ can be computed by,
\begin{equation}\label{expressionofsigma}
 \sigma_k\left(p_s\right)=\left [C_d\left(p_s\right)\int_0^\infty{e^{J_d\left(p_s\right) t}K_d\left(p_s\right) \:K_{d}\left(p_s\right)^\top e^{J_d\left(p_s\right)^{\top} t}}dt~C_d\left(p_s\right)^\top\right]^{\frac{1}{2}}\left(k,k\right),
 \end{equation}
where, the matrix functions $C_d,\:J_d,\:K_d$ in Equation (\ref{expressionofsigma}) are infinitely differentiable with respect to $p_s\in P^+$, hence $\sigma_k$ is infinitely differentiable on $P^+$. 
\end{proof}
\begin{lemma}\label{VaLC}
The standard deviation of the phase-angle differences, denoted as $\sigma: P^+ \to \mathbb{R}_{s,+}^{n_{E}}$, is Lipschitz continuous on $P^+$.
\end{lemma}
\begin{proof}
The gradient of the function $\sigma$ for a vector $p_s\in P^+$, denoted as $\nabla \sigma(p_s)\in \mathbb{R}^{n_{E}\times (n^+-1)}$, is a continuous matrix function on the compact domain $P^+$. This continuity is guaranteed by the fact that $\sigma$ is a smooth function on $P^+$. Moreover, due to the compactness of the domain $P^+$, each component of $\nabla \sigma(p_s)$ can be bounded, which leads to $\|\nabla \sigma(p_s)\|_{2,s} \leq \|\nabla \sigma(p_s)\|_F \leq c_1$, where $c_1\in\mathbb{R}_{s,+}$ is a constant. Based on the results in \cite[p.15]{khalil2015nonlinear}, we can conclude that the function $\sigma$ is Lipschitz continuous on $P^+$.
\end{proof}
\begin{lemma}\label{PhDiLC}
The absolute value of the mean of the phase-angle differences of the entire power network, denoted as the function $f_{as}$ with the formula for a vector $p_s\in P^+$, $f_{as}\left(p_s\right) = \arcsin\left(|A\:p_s+b|\right)$ in Definition \ref{def:controlobjectivefunction}, is Lipschitz continuous on the domain $P^+$.
\end{lemma}
\begin{proof}
Define a function $e: P^+ \to \mathbb{R}^{n_{E}}$, for a vector $p_s\in P^+,\:e\left(p_s\right) = \arcsin\left(A\:p_s+b\right)$, representing the mean of the phase-angle differences for the entire power network.
Consider two vectors $p_a,\:p_b\in P^+$. For $i\in \mathbb{Z}_{n_E}$, define the vectors $\zeta_i = \theta_i\: p_a + (1-\theta_i)\: p_b$, where $\theta_i\in (0,1)$ , according to the Mean Value Theorem for the function $e$.
The Jacobian matrix of the function $e$, denoted as $J_e$, is defined as:
$J_e=\begin{pmatrix}\nabla e_1\left(\zeta_1\right);& \nabla e_2\left(\zeta_2\right);&\cdots&;\nabla e_{n_E}\left(\zeta_n\right)\end{pmatrix}\in \mathbb{R}^{n_{E}\times \left(n^+-1\right)}$.
Recall the notations, the i-th row of the matrix $J_e$ is denoted as $J_{e}\left(i\right)$,
\begin{eqnarray*}
    \lefteqn{\|f_{as}\left(p_a\right)-f_{as}\left(p_b\right)\|_2}\\
    & =& \|\arcsin\left(|A\:p_a+b|\right)-\arcsin\left(|A\:p_b+b|\right)\|_2\\
    & \leq& \|\arcsin\left(A\:p_a+b\right)-\arcsin\left(A\:p_b+b\right)\|_2\\
    & =& \|J_e\: \left(p_a-p_b\right)\|_2 \leq \left(\sum_{i=1}^{n_E}{\|J_e(i)\|_2^2}\right)^{\frac{1}{2}}\cdot \|p_a-p_b\|_2,\\
    &:=&K_1\cdot \|p_a-p_b\|_2,
\end{eqnarray*}
where the inequality of the third line is according to the Cauchy–Schwarz inequality.
\end{proof}
\begin{proposition}\label{ObjectLC}
The control objective function $f$, 
as defined in Definition~\ref{def:controlobjectivefunction}, 
is Lipschitz continuous on the domain $P^+$.
\end{proposition}
\begin{proof}
Consider two vectors $p_a,p_b\in P^+$. Furthermore, we denote the Lipschitz constant of the function $\sigma$ as $K_2$, which satisfies $\|\sigma\left(p_a\right)-\sigma\left(p_b\right)\|\leq K_2\cdot \|p_a-p_b\|$,
\begin{align*}
  \lefteqn{
            \|f\left(p_a\right)-f\left(p_b\right)\|_2=|f\left(p_a\right)-f\left(p_b\right)|
          }\\
    & =\left | \|\arcsin\left(|A\:p_a+b|\right)+r\cdot\sigma\left(p_a\right)\|_{\infty}-\|\arcsin\left(|A\:p_b+b|\right)+r\cdot\sigma\left(p_b\right)\|_{\infty}\right|\\
    & \leq\|\arcsin\left(|A\:p_a+b|\right)+r\cdot\sigma\left(p_a\right)-\arcsin\left(|A\:p_b+b|\right)-r\cdot\sigma\left(p_b\right)\|_{\infty}\\
    & \leq \|\arcsin\left(|A\:p_a+b|\right)-\arcsin\left(|A\:p_b+b|\right)\|_{\infty}+\|r\cdot\sigma\left(p_a\right)-r\cdot\sigma\left(p_b\right)\|_{\infty}\\
    & \leq K_1 \cdot \|p_a-p_b\|_2+\|r\cdot\sigma\left(p_a\right)-r\cdot\sigma\left(p_b\right)\|_{\infty}\leq K_1  \cdot\|p_a-p_b\|_2+r\cdot K_2\cdot \|p_a-p_b\|_{\infty}\\
    & \leq K_1 \cdot \|p_a-p_b\|_2+r\cdot K_2\cdot \|p_a-p_b\|_{2}
      =(K_1+r\cdot K_2)\cdot\|p_a-p_b\|_{2}~ :
      = G\cdot\|p_a-p_b\|_{2}.
\end{align*}
\end{proof}
\subsection{A property of the Lipschitz continuous function}\label{prooffun}
We establish a property of the Lipschitz continuous function, demonstrating the Step $2$ of Theorem \ref{existence}
which in turn, leads to the derivation of Steps 3 and 4 of Theorem \ref{existence}.
\begin{proposition}\label{ComconSet}
Consider a function $ f:\mathbb{R}^{m}\rightarrow \mathbb{R}^{n}$ 
which is Lipschitz continuous. 
If the domain of definition of $f$ denoted by $P^+$, is a compact convex set, 
then the value range of $f$ denoted by $\Range\left(f\right)$ is also a compact convex set. 
\end{proposition}
\begin{proof}
 If the domain of a continuous function is compact, then the value range of that function is also compact, indicating the compactness of $\Range\left(f\right)$. To establish its convexity, we begin by examining its converse proposition, which can be articulated as follows: Consider two vectors $p_a$ and $p_b$, belonging to the domain $P^+$. There exists a $y\in \mathbb{R}^{n}$ and a $\theta\in\left(0,1\right)$ which satisfies $y=\theta\cdot f\left(p_a\right)+\left(1-\theta\right)\cdot f\left(p_b\right)$, such that $f\left(p_s\right)\:\neq \:y, \forall \:p_s\:\in P^+$.

For this vector $y$, it satisfies $\|y-f\left(p_b\right)\|_2=\theta\cdot\|f\left(p_a\right)-f\left(p_b\right)\|_2$ and $\|y-f\left(p_a\right)\|_2=\left(1-\theta\right)\cdot\|f\left(p_b\right)-f\left(p_a\right)\|_2$. Leveraging the Lipschitz continuity property of the function $f$, we can derive $\|f\left(p_b\right)-f\left(p_a\right)\|_2\leq G\cdot\|p_b-p_a\|_2$, with $G\in \mathbb{R}_{s,+}$ denoted as the Lipschitz constant as in Proposition \ref{ObjectLC}. Moreover, owing to the convexity of the domain $P^+$, the distance between vectors $p_a$ and $p_b\in P^+$, denoted as $\|p_b-p_a\|_2$, can be made arbitrarily small. By setting this distance less than a small number $\delta\in\mathbb{R}_{s+},\:\|p_b-p_a\|_2<\delta$, we have $\|f\left(p_b\right)-f\left(p_a\right)\|_2\leq G\cdot\delta$. Consequently, $\|y-f\left(p_a\right)\|_2\leq\left(1-\theta\right)\cdot G\cdot\delta$ and $\|y-f\left(p_b\right)\|_2\leq\theta\cdot G\cdot\delta$.

As a result, the two vectors $f\left(p_a\right),\:f\left(p_b\right)\in \mathbb{R}^{n}$ are inside a spherical neighborhood denoted by $B\left(y,\:G\cdot\delta\right)$, with a center at $y$ and a radius of $G\cdot\delta$. However, since $\forall\:p_s\in P^+,\:f\left(p_s\right)\neq y$, the continuity property of the function $f$ ensures that there is no $y'\in \Range\left(f\right)$ such that $y'\in B\left(y,\:G\cdot\delta\right)$. This contradicts the fact that $\Range\left(f\right)$ contains at least two vectors, $f\left(p_a\right)$ and $f\left(p_b\right)$ in this spherical neighborhood. Therefore, the converse proposition is incorrect, establishing the convexity of the value range $\Range\left(f\right)$.
\end{proof}
\section{Convexity and differentials of the control objective function}\label{sec:convexityanddifferential}
\subsection{Convexity of the control objective function}\label{basicproperty}
Is the control objective function convex and differentiable? 
Readers interested in exploring optimization and convexity
can refer to the books
\cite{rockafellar1970convex,rockafellar:wets:2009,boyd2004convex,nesterov2018lectures}. 
%
\begin{lemma}\label{gconvex}
The function 
$f_{as,k}:P^+\rightarrow \mathbb{R}_+,\forall~ k\in \mathbb{Z}_{n_{E}}$, is convex on $P^+$. 
\end{lemma}
\begin{proof}
The function $f_{as,k}$ for a vector $p_s\in P^+$, defined as $f_{as,k}\left(p_s\right) = \arcsin \left(|(A \:p_s + b)_k|\right)$\\$= \arcsin \left(|A\left(k\right)\:p_s + b_k|\right)$. 
This function is differentiable 
except on the subset \\$\left\{p_s\in P^+ : A(k)\: p_s + b_k = 0\right\}\subseteq P^+$. 
The expression of Hessian matrix of $f_{as,k}$ for a vector 
$p_s \in \left\{p_s\in P^+ : A(k) \:p_s + b_k\neq 0\right\}$ 
is given by
$\nabla^2f_{as,k}(p_s) = \left(1 - (A(k) \:p_s + b_k)^2\right)^{-3/2}\cdot |A(k) \:p_s + b_k| 
 \cdot A(k)^\top A(k)$, 
which is always positive definite, though not necessarily strictly positive definite. 
This positive definiteness establishes the convexity property of the function $f_{as,k},\:\forall~ k\in \mathbb{Z}_{n_{E}}$.
\end{proof}
\begin{lemma}\label{snotconvex}
In general, 
each component of the standard deviation of the phase-angle differences, 
denoted as $\sigma_k: P^+ \to \mathbb{R}_{s,+}^{n_{E}},\:\forall\:k\in \mathbb{Z}_{n_{E}}$, 
is not convex.
\end{lemma}
\begin{example}
Note that whether the standard deviation of the phase-angle difference of power line $k$, denoted as $\sigma_k$, is convex or not cannot be conclusively determined from Equation (\ref{expressionofsigma}). To prove Lemma \ref{snotconvex}, we have developed an algorithm to compute the gradient vector and the Hessian matrix of the function $\sigma_k$ numerically by using the directional derivative method. The algorithm can be found in Section \ref{gradientandHessianofS}.
\par
In our computational analysis, we have examined several cases, and our findings indicate that the standard deviation $\sigma_k,\:\forall\:k\in \mathbb{Z}_{n_{E}}$ is not convex in all cases. This is primarily due to the presence of both strictly positive and strictly negative eigenvalues in the Hessian matrix.
\end{example}
Consequently, the control objective function $f$ on $P^+$ can not be guaranteed to be convex, which means it may possess multiple local minimizers. To ensure that each local minimizer of the control objective function is isolated, we introduce the following assumption.
\begin{assumption}\label{assumption:noneighborhoodzero}
There does not exists an open neighborhood $\mathcal{O}\subset P^+$ for which the control objective function $f$ remains constant on this neighborhood.
\end{assumption}

\subsection{A partition of the set of power supply vectors}\label{sec:partition}
The definition of the partition is preceeded
by concepts of algebraic geometry.
\par
In geometry one describes surfaces.
Ways to specify surfaces include
hyperplanes described by affine functions;
algebraic sets described by polynomials; and
other surfaces described by functions
which are not polynomials.
For power systems, all three cases appear.
Terminology of a surface described by a polynomial follows,
\cite[Ch.1, paragraph 2, Def. 1]{Cox1992}.
\par
Consider a polynomial in $n \in \mathbb{Z}_+$ indeterminates
of degree $d \in \mathbb{Z}_+$ as an algebraic object
\begin{eqnarray*}
    &   & \forall ~ n \in \mathbb{Z}_+, ~
          \forall ~ d \in \mathbb{Z}_+, ~
          \forall ~ k \in \mathbb{N}_d^n \left(k\in  \mathbb{R}^n, k(i)\in  \mathbb{N}_d\right) ~
          \mbox{define the monomial,} \\
        p^k
    & = & \prod_{i=1}^n ~ p_i^{k(i)}; ~
          \mbox{define the polynomial for any finite subset} ~
          \mathbb{N}_s \subset \mathbb{N}_d^n, \\
        q(p)
    & = & \sum_{k \in \mathbb{N}_s} ~ c(k) ~ p^k \in 
          \mathbb{R}[p_1, ~ p_2, ~ \ldots, ~ p_n ], ~
          \mbox{where, } ~ \forall ~ k\in \mathbb{N}_s, ~ c(k) \in \mathbb{R}; \\
    &   & \forall ~ Q_s \subset 
          \mathbb{R}[p_1, ~ p_2, ~ \ldots, ~ p_n ], ~
          \mbox{a finite subset of polynomials, define,} \\
        V(\mathbb{R}^n, ~ Q_s)
    & = & \left\{
            p \in \mathbb{R}^n |~ \forall ~ q \in Q_s, ~ q(p) = 0 
          \right\}.
\end{eqnarray*}
Call $V(\mathbb{R}^n, ~ Q_s)$
an {\em algebraic set} or an {\em affine variety}.
Call a subset $G \subseteq \mathbb{R}^n$
a {\em generic subset} of $\mathbb{R}^n$
if $\mathbb{R}^n \backslash G$ is an algebraic set.
\par
This is now applied to the control objective function.
Define,
            $f_{lin,k}\left(p_s\right)
	= | A\left(k\right)\:p_s+ b_k|, ~
	      \forall ~ k \in \mathbb{Z}_{n_{E}}.$
Recall the notation, $A\left(k\right)$ denotes the k-th row of matrix $A$ and $A\left(k\right) p_s+ b_k$ is an affine function of $p_s$,
thus a polynomial, hence
$V(\mathbb{R}^{n^+-1}, ~f_{as,k})$ is an algebraic set.
Note that, for $x \in (-1, ~ +1)$,
$\arcsin(x) = 0$ if and only if $x=0$.
However,
$f_k\left(p_s\right) = \left[\arcsin\left(|A\left(k\right) p_s + b_k|\right)\right] + r\cdot\sigma_k\left(p_s\right)$ defined in
Def. \ref{def:controlobjectivefunction}
is not a polynomial because
the function arcsin is not a polynomial.
It can be proven that $\sigma_k(p_s)$ is a polynomial 
in terms of the components of $p_s$.
\begin{definition}\label{def:cases}
Consider the control objective function of
Def.~\ref{def:controlobjectivefunction}.
Distinguish the cases:
\begin{itemize}
\item
Case 1.
there exists a unique $k \in \mathbb{Z}_{n_E}$ and
there exists a nonempty subset $P_{(k)}^+$ such that 
$P_{\left(k\right)}^+ = \{ p_s \in P^+ |~ f\left(p_s\right) = f_k\left(p_s\right) \}$.
\item
Case 1.1. 
there exists a unique $k \in \mathbb{Z}_{n_E}$ and
there exists a nonempty subset $P_{(k),nz}^+$ such that
$P_{\left(k\right),nz}^+ 
= \{ p_s \in P_{\left(k\right)}^+ |~ f\left(p_s\right) = f_k\left(p_s\right), ~ f_{as,k}\left(p_s\right) \neq 0 \}$.
\item
Case 1.2. 
there exists a unique $k \in \mathbb{Z}_{n_E}$ and
there exists a nonempty subset $P_{\left(k\right),z}^+$ such that
$P_{\left(k\right),z}^+ 
= \{ p_s \in P_{\left(k\right)}^+ |~ f\left(p_s\right) = f_k\left(p_s\right), ~ f_{as,k}\left(p_s\right) = 0 \}$.
\item
Case 2.
there exists two or more
$k_1, ~ k_2, ~ \ldots, ~ k_m \in \mathbb{Z}_{n_E}$
and there exists a nonempty subset 
$P_{\left(k_1, k_2, \ldots, k_m\right)}^+$
such that\\
$P_{\left(k_1, k_2, \ldots, k_m\right)}^+
= \{ p_s \in P^+ |~ f\left(p_s\right) = f_{k_1}\left(p_s\right) = f_{k_2}\left(p_s\right) = \ldots = f_{k_m}\left(p_s\right) \}$.\\
Denote by
$I_{max}\left(p_s\right) = \{ k_1, ~ k_2, ~ \ldots, ~ k_m \}$
the subset of those integers.
\end{itemize}
\end{definition}
The subset
$P_{\left(k_1,k_2\right)}^+$ 
may not be an algebraic set
because the relation 
$f_{k_1} \left(p_s\right)= f_{k_2}\left(p_s\right)=\left[\arcsin\left(A(k_2)\:p_s+b_{k_2}\right)\right]+r\cdot\sigma_{k_2}\left(p_s\right)$ is not a polynomial in general.
It is clear from the formulas,
that Case 1.1 is the generic case and that Case 1.2 takes place on an algebraic set.
\subsection{The first directional derivative of the control objective function}\label{firstorderdiff}
Books about the directional derivatives
and related concepts include
\cite{rockafellar:wets:2009,nesterov2018lectures,cui2021modern}.
\begin{definition}\label{def.firstorderderivative}
\cite[Def. 3.1.3]{nesterov2018lectures}.
\par
Consider integers $m, ~ n \in \mathbb{Z}_+$,
a convex and open subset $U \subseteq X = \mathbb{R}^n$, and
a function $g: U \rightarrow \mathbb{R}^m$.
Assume that:
(1) the function $g$ is continuous on its domain of definition $U$; and
(2) there does not exist an open subset $O \subseteq U$
on which the function $g$ is constant.
\par
One says that the function $g$ is 
{\em directionally differentiable at $x_s \in U$
in the direction $v \in \mathbb{R}^n$} 
if there exists a linear map
$L: \mathbb{R}^n \rightarrow \mathbb{R}^m$
such that the following limit exists,
\begin{eqnarray}
        L\left(v\right)
    & = & \lim_{t \in \mathbb{R}_{s+}, ~ t \downarrow 0} ~
          \frac{g\left(x_s + t \cdot v\right) - g\left(x_s\right)}{t}, ~
          \Leftrightarrow ~ \nonumber \\
        0
    & = & \lim_{t \in \mathbb{R}_{s+}, ~ t \downarrow 0} ~
          \frac{g\left(x_s + t  \cdot v\right) - g\left(x_s\right) - t  \cdot L\left(v\right)}{t}; ~
          \mbox{denote then,} \\
        dg\left(x_s, ~ v\right)
    & = & L\left(v\right), ~ \forall ~ v \in \mathbb{R}^n.
\end{eqnarray}
Call then $dg\left(x_s, ~ v\right)$ the
{\em directional differential of $g$ at $x_s$ in the direction $v$}.
\par
If a basis of the vector space $\mathbb{R}^n$ has been chosen 
to be $\left\{ e_1, ~  e_2, ~ \ldots, ~ e_n \right\}$,
the set of the Euclidian unit vectors, 
then the linear map can be represented by the
Jacobian matrix, 
\begin{eqnarray*}
               J_g\left(x_s\right)
    & = & g'\left(x_s\right) =
          \begin{bmatrix}
            L\left(e_1\right) & L\left(e_2\right) & \ldots & L\left(e_n\right)
          \end{bmatrix}
          \in \mathbb{R}^{m \times n}; ~
          \mbox{equivalently,} \\
        J_g\left(x_s\right)_{i,j}
    & = & \lim_{t \in \mathbb{R}_{s+}, ~ t \downarrow 0} ~
          \frac{g_i\left(x_s + t  \cdot e_j\right) - g_i\left(x_s\right)}{t}, ~
          \forall ~ i \in \mathbb{Z}_m, ~ \forall ~ j \in \mathbb{Z}_n; ~
          \mbox{then} \\
        dg\left(x_s, ~ v\right)
    & = & L\left(v\right) = J_g\left(x_s\right) ~ v = g'\left(x_s\right) ~ v, ~
          \forall ~ v \in \mathbb{R}^n.
\end{eqnarray*}
Call the matrix $J_g\left(x_s\right)$
the {\em Jacobian matrix} of $g$ at $x_s \in U$.
\end{definition}
In general, for Case 1.2 and Case 2,
the Jacobian matrix does not exist
for the control objective function
of Def.~\ref{def:controlobjectivefunction}.
Even if the Jacobian matrix for a particular direction exists
then the Jacobian matrix can be different for another direction.
Fro Case 2, the directional differential $df\left(p_s, ~ v\right)$
will not be a continuous function of the direction vector $v$.
A further formalization of a sectorwise directional derivative,
is not stated in this paper. 
\begin{proposition}\label{prop:firstdirectionalderivative}
Consider the control objective function of
Def.~\ref{def:controlobjectivefunction}.
\begin{itemize}
\item[(a)]
Case 1.1.
For any $p_s \in P_{\left(k\right),nz}^+ $,
and a direction vector $v \in \mathbb{R}^{n^+-1}$, 
such that $p_s+v \in P^+$,
the  first directional derivative of the
control objective function exists and equals,
\begin{eqnarray*}
    \lefteqn{
      f'\left(p_s, ~ v\right) = f_k'\left(p_s,v\right)=\nabla f\left(p_s\right) v=\left(\nabla f_{as,k}\left(p_s\right)+r \cdot\nabla\sigma_k\left(p_s\right)\right) v
    } \\
      & = &\left\{
	    \begin{array}{ll}
           \left[ \left(
              1 - \left(A\left(k\right) \:p_s + b_k\right)^2 
            \right)^{-\frac{1}{2}}\cdot
            A\left(k\right) + r \cdot\nabla\sigma_k\left(p_s\right)\right]v, 
            & \mbox{if} \:A\left(k\right)~ p_s + b_k > 0, \\
          \\
           \left[ -
	    \left(
              1 - \left(A\left(k\right) p_s + b_k\right)^2 
            \right)^{-\frac{1}{2}}\cdot
            A\left(k\right) + r\cdot \nabla\sigma_k\left(p_s\right)\right]v, 
	    & \mbox{if} ~A\left(k\right)~p_s + b_k < 0.
	    \end{array}
	    \right. \nonumber
\end{eqnarray*}
\item[(b)]
Case 1.2.
For any $p_s \in P_{(k),z}^+$,
and a direction vector $v \in \mathbb{R}^{n^+-1}$, 
such that $p_s+v \in P^+$,
the first directional derivative of the
control objective function exists and equals,
\begin{eqnarray*}
    f'\left(p_s, ~ v\right)
    & = & f_k'\left(p_s\right)
          = | A\left(k\right) ~ v | + r \cdot\nabla\sigma_k\left(p_s\right)\:v 
                \end{eqnarray*}
Note that in this case, 
the Jacobian matrix depends on the direction vector $v$.
\item[(c)] \cite[Ch. 4, p. 157]{cui2021modern}.
Case 2.
For any $p_s \in P_{\left(k_1,k_2,\cdots,k_m\right)}^+$,
and a direction vector $v \in \mathbb{R}^{n^+-1}$, 
such that $p_s+v \in P^+$, the first directional derivative of the
control objective function exists and is equal to,
$$f'\left(p_s, ~ v\right) = \max_{k \in I_{max}\left(p_s\right)} ~ f_k'\left(p_s, ~ v\right)$$
Note that $f_k'\left(p_s,v\right)$ takes the expression of Case 1.1 or Case 1.2, depending on whether  $A\left(k\right) p_s + b_k$ is nonzero or zero, respectively.
\end{itemize}
\end{proposition}
\begin{proof}
\par For Case 1.1, $\forall ~p_s\in P_{1,k}^+$ where the control objective function $f(p_s)$ is differentiable,
\begin{eqnarray*}
\lefteqn{f'\left(p_s,v\right)= f_k'\left(p_s,v\right)}\\
&=&\left(1-\left(A\left(k\right)p_s +b_k\right)^2\right)^{-\frac{1}{2}}\cdot \left(A\left(k\right)\:v\right)\cdot sign\left(A\left(k\right)\:p_s+b_k\right)+r\cdot\nabla \sigma_k\left(p_s\right)\:v,
\end{eqnarray*}
It is obvious that the first directional derivative of the control objective function is a linear function of the direction $v$ in this case
\par
For Case 1.2, $\forall ~p_s\in P_{2,k}^+$, we apply Definition \ref{def.firstorderderivative},
\begin{align*}
    \lefteqn {
      f'\left(p_s,v\right)
      = f_k'\left(p_s,v\right)
      = \lim\limits_{
	  \tau\downarrow 0}{
	    \frac{
	    \left\{
            \begin{array}{l}
              \arcsin\left(|A\left(k\right)\left(p_s+\tau\cdot  v\right)+b_k|\right)
	      + r \cdot \sigma_k\left(p_s+\tau\cdot  v\right)\\
              -\left[\arcsin\left(|A\left(k\right)\:p_s+b_k|\right)
              + r\cdot\sigma_k \left(p_s\right)\right]
            \end{array}
	    \right\}}%
            {\tau}}
       } \\
  & = \lim\limits_{\tau\downarrow 0}{\frac{\arcsin(|\tau\cdot A\left(k\right)\:v|)+r\cdot\sigma_k\left(p_s+\tau\cdot  v\right)-r\cdot\sigma_k\left(p_s\right)}{\tau}}\\
  & =\lim\limits_{\tau\downarrow 0}{\frac{\arcsin\left(|\tau\cdot  A\left(k\right) v|\right)}{\tau}}+r\cdot\nabla \sigma_k\left(p_s\right)\:v =|A\left(k\right) \:v|+r\cdot\nabla \sigma_k\left(p_s\right)\:v, ~
	\qquad\qquad\qquad\quad
\end{align*}
where, the first equality of the last line is because of $\sigma_k\in C^{\infty}\left(P^+\right)$ and the second equality of that line is by L'Hospital's rule. The first directional derivative of the control objective function $f(p_s,v)$ is a piecewise linear function of the direction $v$ in this case.
\end{proof}
For Case 2 and $m=2$, one can distinguish 9 subcases.
Most subcases have two or four different expressions for the Jacobian matrices.
Because Case 2 will occur less frequently than Case 1.2,
the authors have decided not to include the formulas
of the Jacobian matrices for each subcase.
\subsection{The second directional derivative of the control objective function}\label{secondorderdiff}
\begin{definition}\label{directionalderivative2}
\cite[Ch. 4.2, p. 162]{cui2021modern}
Let $f:\mathcal{O}\subset \mathbb{R}^{n^+-1}\to  \mathbb{R}$ 
be a scalar-valued function defined on an open subset 
$\mathcal{O}$ of  $\mathbb{R}^{n^+-1}$ 
that is directional differentiable at $x$. 
The {\em second directional derivative} 
of $f$ at $x\in \mathcal{O}$ in a direction $v\in \mathbb{R}^{n^+-1}$ 
exists if the following limit exists,
\begin{align*}
     \frac{1}{2}f''\left(x,v\right)
     & := \lim\limits_{\tau\downarrow 0}
       {\frac{f\left(x+\tau\cdot  v\right)
	-f\left(x\right)
	- \tau \cdot f'\left(x,v\right)}{\tau^2}
       }.
\end{align*}
We say that $f$ is twice directionally differentiable at $x$ 
if it is directionally differentiable at  $x$ and 
if the limit $f''(x,v)$ exists for all $v \in\mathbb{R}^{n^+-1}$.
\par 
\cite[Ch. 4.2, formula 4.11, p. 165]{cui2021modern}
A function f is said to be twice semidifferentiable at $x$ if it is directional differentiable at $x$ and the limit
\begin{displaymath}
    \frac{1}{2} f''
    \left(
      x, v
    \right)
    :=\lim\limits_{\substack{v'\to v\\ \tau\downarrow 0}} ~
	{\frac{f\left(x+\tau\cdot  v'\right)-f\left(x\right)-\tau\cdot  f'\left(x,v'\right)}{\tau^2}}, ~
	\mbox{exists} ~ \forall ~ v \in \mathbb{R}^n.
\end{displaymath}
\end{definition}
\begin{proposition}\label{prop:seconddirectionalderivative}
Consider the control objective function of
Def.~\ref{def:controlobjectivefunction}.
\begin{itemize}
\item[(a)]
Case 1.1.
For any $p_s \in P_{\left(k\right),nz}^+$,
and a direction vector $v \in \mathbb{R}^{n^+-1}$, 
such that $p_s+v \in P^+$,
the second directional derivative of the
control objective function exists and equals,
\begin{align*}
    \frac{1}{2}\:f''\left(p_s, v\right)
    & = \frac{1}{2}\: f_{k}''\left(p_s, v\right)
      = \left(
              1 - 
	      \left(
	        A
		\left(
		  k
	        \right) \: p_s 
		+ b_k
              \right)^2 
            \right)^{-\frac{3}{2}} \times
	\\
    & ~~~~~ \times
      | A\left(k\right)\:p_s + b_k | \cdot 
      v^{\top} A\left(k\right)^{\top} A\left(k\right) v
      + r \cdot v^{\top} \nabla^2\left(\sigma_k\left(p_s\right)\right) v.  
\end{align*}
\item[(b)]
Case 1.2.
For any $p_s \in P_{\left(k\right),z}^+$,
and a direction vector $v \in \mathbb{R}^{n^+-1}$, 
such that $p_s+v \in P^+$,
the second directional derivative of the
control objective function exists and is equals,
\begin{eqnarray*}
    \frac{1}{2} f''\left(p_s, v\right)
    & = &\frac{1}{2}  f_{k}''\left(p_s, ~ v\right) 
          = v^{\top} \:\nabla^2 \sigma_k\left(p_s\right) v.
\end{eqnarray*}
\item[(c)] Case 2. \cite[Ch. 4, p. 157]{cui2021modern}.
For any $p_s \in P_{\left(k_1,k_2,\cdots,k_m\right)}^+$,
and a direction vector $v \in \mathbb{R}^{n^+-1}$, 
such that 
$p_s+v \in P^+$, and define $I_{max}\left(p_s,v\right)
:=\{k\in I_{max}\left(p_s\right)~|~f'\left(p_s,v\right)=f_k'\left(p_s,v\right)\}$,
the second directional derivative of the
control objective function exists and equals,
\begin{eqnarray*}
  \frac{1}{2} f''\left(p_s,v\right)
  = \frac{1}{2}\max\limits_{k\in I_{max}\left(p_s,v\right)}{f_k''\left(p_s,v\right)}.
\end{eqnarray*}
Note that $f_k''\left(p_s,v\right)$ takes the expression of Case 1.1 or Case 1.2, depending on whether  $A\left(k\right) p_s + b_k$ is nonzero or zero, respectively.
\end{itemize}
\end{proposition}
The proof is omitted
because it is similar to the proof of
Proposition~\ref{prop:firstdirectionalderivative}.
\begin{lemma}\label{objectnottwicesemidiffe}
The control objective function $f(p_s)$ is not necessarily twice semidifferentiable. 
\begin{proof}
By \cite[Ch. 4.2, example 4.2.1, p. 165]{cui2021modern}
\begin{equation*}
\begin{aligned}
&\frac{1}{2}\min\limits_{k\in I_{max}\left(p_s,v\right)}{f_k''\left(p_s,v\right)}\leq 
\liminf\limits_{\substack{v'\to v\\ \tau\downarrow 0}}{\frac{f\left(p_s+\tau\cdot  v'\right)-f\left(p_s\right)-\tau\cdot  f'\left(p_s,v'\right)}{\tau^2}}\leq\\
& \limsup\limits_{\substack{v'\to v\\ \tau\downarrow 0}}{\frac{f\left(p_s+\tau\cdot  v'\right)-f\left(p_s\right)-\tau\cdot  f'\left(p_s,v'\right)}{\tau^2}}\leq
\frac{1}{2}\max\limits_{k\in I_{max}\left(p_s,v\right)}{f_k''\left(p_s,v\right)}.
\end{aligned}
\end{equation*}
By formulas in Proposition \ref{prop:seconddirectionalderivative}, $\min\limits_{k\in I_{max}\left(p_s,v\right)}{f_k''(p_s,v)}\neq \max\limits_{k\in I_{max}\left(p_s,v\right)}{f_k''(p_s,v)}$, in general.
\end{proof}
\end{lemma} 
\begin{remark}
Note that due to the non-necessarily twice semidifferentiable characteristics of the control objective function $f$, we cannot establish the continuity of $f''(p_s, v)$ concerning $v \in \mathbb{R}^{n^+-1}$. Hence, we can not make sure a vector $p_s^*\in P^+$ which satisfies $f'(p_s^*,v)\geq 0,~\forall~v\in \mathbb{R}^{n^+-1},\:p_s^*+v\in P^+$ and $f''(p_s^*,v)\geq 0$, for all $v$, such that $p_s^*+v\in P^+,\:f'(p_s^*,v)= 0$ , is a local minimizer. Instead we should also investigate the vectors in the neighborhood of $p_s^*$.
\end{remark}
\section{Algorithms, the descent method, and convergence theorems}\label{sec:algorithms}
The gradient vector and Hessian matrix of the implicit functions $V_i$ and $\sigma_i$, $i\in\mathbb{Z}_{n_E}$
 defined in Procedure~\ref{proc:computationcontrolobjectivefunction}, Step 3
are analyzed in Subsections~\ref{gradientandHessianofV} and \ref{gradientandHessianofS}, respectively.
Subsections~\ref{sec:steepest} 
and \ref{subsec:convergence}
introduce an algorithm using the steepest descent method 
for approximating a local minimizer and 
the convergence analysis theorem.
In Subsection \ref{sec:computationofinitialvector},
an algorithm is provided to compute an approximate stationary vector
to serve as an initial vector for the algorithm of Subsection~\ref{sec:steepest}.
\subsection{The analysis and computation of the gradient vector and the Hessian matrix of a {\em variance} $V_i$, with respect to a power supply vector $p_s\in P^+$}\label{gradientandHessianofV}
Since the variance $V_i$ does not have an explicit formula, we decide to use
the directional derivative method similar as \cite[Theorem 6]{optimal_inertia_placement} to compute the gradient vector and the Hessian matrix of that function.
\begin{proposition}\label{directionalmethod}
The first and the second directional derivatives 
of the implicit function $V_i:P^+\rightarrow \mathbb{R}_+,\:i\in\mathbb{Z}_{n_E}$, 
of a vector $p_s\in P^+$,
can be computed according to the formulas,
\begin{align*}
    V_i\left(p_s+\delta\:\mu\right)
    & = V_i^{(0)}\left(p_s\right)+V_i^{(1)}\left(p_s\right)\delta+V_i^{\left(2\right)}\left(p_s\right)\delta^2,\:\delta\in\mathbb{R},\:\mu\in \mathbb{R}^{n^+-1},\:\text{where,}\\
    V_i^{(1)}\left(p_s\right)
    & = \nabla_\mu V_i\left(p_s\right)
      =\mu^\top \nabla V_i\left(p_s\right), ~~~
      V_i^{(2)}\left(p_s\right)
      = \frac{1}{2}\:\mu^\top \nabla^2 \:V_i\left(p_s\right)\:\mu.
\end{align*}
\end{proposition}
Using this proposition, specifying $k,\:j\in \mathbb{Z}_{n^+-1}$,
one can calculate $\nabla V_i\left(p_s\right)\left(k\right)$ 
by $\mu=e_k$ and
$\frac{1}{2}\cdot\left(\nabla^2 V_i\left(p_s\right)\left(k,k\right)+\nabla^2 V_i\left(p_s\right)\left(j,j\right)+\nabla^2 V_i\left(p_s\right)\left(k,j\right)+\nabla^2 V_i\left(p_s\right)\left(j,k\right)\right)$ 
by $\mu=e_k+e_j$.
\par
However, after perturbing the variable from $p_s$ to $p_s+\delta\:\mu$, one needs to differentiate each of the matrices in Equation (\ref{expressionofsigma}) with respect to $\delta$. The weight matrix function $W$, as mentioned in Lemma \ref{Vaindiff}, depends on the variable $p_s$ and is part of the system matrix $J_d$ in Equation (\ref{Systemmatrix3}). Therefore, it should be investigated first. View the matrix $W(p_s+\delta\mu)$ as a matrix function with respect to $\delta$ and define $h(\delta)=W(p_s+\delta\mu)$. 
Its Taylor expansion equals,
\begin{equation*}
\begin{aligned}
h\left(\delta\right)&=W\left(p_s+\delta\:\mu\right)=W^{(0)}+W^{(1)}\:\delta+W^{(2)}\:\delta^2, \:\text{where,}\\
W^{(0)}&=W\left(p_s\right),\:W^{(1)}={\frac{dh\left(\delta\right)}{d\delta}}|_{{\delta=0}};\:W^{(2)}=\frac{1}{2}\frac{d^2h(\delta)}{d\delta^2}|_{\delta=0}.
\end{aligned}
\end{equation*}
In the formulas above, the matrices $W^{(1)}$ and $W^{(2)}$ 
need to be computed from matrices that are already available. 
First, we define a matrix $E$ as 
$E=[U_1-U_{n^+}, \cdots,U_{n^+-1}-U_{n^+},U_{n^++1}-U_{n^+},
  \cdots,U_{n_{E}}-U_{n^+}]\in\mathbb{R}^{n_V\times \left(n_V-1\right)}$, 
with respect to the orthonormal matrix $U$ as defined in 
Procedure~\ref{proc:computationcontrolobjectivefunction} Step 1.
Next, recall the notation $E_i\in\mathbb{R}^{n_V}$, 
representing the i-th column of the matrix $E$. 
The formulas for the matrices $W^{(1)}$ and $W^{(2)}$ 
are displayed in Equation (\ref{W1W2}).
\par
Secondly, we present the Taylor expansions of a variance $V_i\left(p_s+\delta\:\mu\right)$, the i-th row of output matrix $C_{d}\left(\delta\right)$ denoted as $C_{d,i}\left(\delta\right)$, and the solution of the Lyapunov equation $Q\left(\delta\right)$ with respect to $\delta$ in Equation (\ref{ViCdQVi}). Following Proposition \ref{directionalmethod}, we provide formulas for the first directional derivative $\nabla_\mu V_i\left(p_s\right)$ and the second directional derivative $\nabla^2_\mu \:V_i\left(p_s\right)$ in Equation (\ref{GradientandHessian}). Within Equation (\ref{GradientandHessian}), the unknown matrices or vectors are $C_{d,i}^{(0)},C_{d,i}^{(1)},C_{d,i}^{(2)},Q^{(0)},Q^{(1)},Q^{(2)}$. Before calculating these matrices, the transformation matrix $U$ as a function of $\delta$, which satisfies~$U(\delta)^{\top}\:M^{-\frac{1}{2}}\:B\:W(p_s+\delta\:\mu)\:B^{\top}\:M^{-\frac{1}{2}}\:U(\delta)=\Lambda_{1}$, must be analyzed. Here, the matrix~$\Lambda_{1}$ is diagonal, and the first diagonal element of the matrix $\Lambda_{1}$ should be $0$. If it is not zero, the columns of the matrix $U(\delta)$ need to be permuted. We describe the Taylor expansion of the transformation matrix function $U$ concerning $\delta$ as follows, similar to the function $h$,
\vspace{-2mm}
\begin{equation*}
\begin{aligned}
 U\left(\delta\right)&=U^{\left(0\right)}+U^{(1)}\:\delta+U^{(2)}\:\delta^2+O(\delta^3), ~\text{where,}\\
 U^{(0)}&=U\left(0\right),\:U^{(1)}=\frac{dU(\delta)}{d\delta} |_{\delta=0},\:U^{(2)}=\:\frac{1}{2}\frac{d^2U(\delta)}{d\delta^2}|_{\delta=0}.
\end{aligned}
\end{equation*}
\vspace{-6mm}
\begin{remark}\label{remarkonU}
This remark pertains to the matrices $U^{(1)}$ and $U^{(2)}$. We define a function, denoted as $L$, with respect to $\delta$ as $L(\delta)=M^{-\frac{1}{2}}\:B\:W(p_s+\delta\:\mu)\:B^{\top}\:M^{-\frac{1}{2}}:\:\mathbb{R}\rightarrow \mathbb{R}^{n_V\:\times\: n_V}$.
\par
If the dimension $n^+-1$ of the matrix $L(\delta)$ is relatively small, say no more than 4, then we may be able to get the analytical solution of the orthonormal matrix $U(\delta)$. The values of the matrices $U^{(0)},\:U^{(1)},\:U^{(2)}$ there follows. 
\par
If the dimension $n^+-1$ of the matrix $L(\delta)$ is large, it follows from Galois' theorem that it is impossible to get the analytical expression of the orthonormal matrix $U(\delta)$. Therefore, a method to approximate the numerical value of the matrices $U^{(0)},\:U^{(1)},\:U^{(2)}$ needs to be formulated. 
It is easy to compute the numerical value of matrix $U\left(0\right)$ by the formula $U\left(0\right)^\top L\left(0\right)U\left(0\right)=\Lambda$, where $\Lambda$ is a diagonal matrix, as the matrix $L\left(0\right)$ is already known. 
The authors then choose the finite difference method 
to approximate the first and second-order derivative matrices 
$U^{(1)},\:U^{(2)}$. 
The errors between the numerical values and real values 
are dependent on the step length denoted by $\Delta$.
For a step length, 
the reader may think of $10^{\left(-4\right)}$ as an example, 
we have the following approximations: 
\vspace{-2mm}
\begin{align*}
     & U^{(1)}
       \approx
       \frac{U\left(\Delta\right)-U\left(-\Delta\right)}{2\:\Delta}, ~
       U^{(2)}
       \approx
       \frac{U\left(\Delta\right)-2\:U\left(0\right)+U\left(-\Delta\right)}{
	     2\:\Delta^2
       },
\end{align*}
where the values of the matrices $U\left(\Delta\right),\:U\left(-\Delta\right)$ can be approximated similarly to $U\left(0\right)$.
\end{remark}
With respect to this remark, the numerical values of the output and the input matrices $C_d^{(0)},\:C_d^{(1)},\:C_d^{(2)},\:K_d^{(0)},K_d^{(1)},\:K_d^{(2)}$ can be computed by the formulas of Equation (\ref{Cd0Cd1Cd2Kd0Kd1Kd2}), and the system matrices $J_d^{(0)},\:J_d^{(1)},\:J_d^{(2)}$ can be computed by Equation (\ref{Ad0Ad1Ad2}). These matrices will be used for the calculation of the solution matrices $\:Q^{(0)},\:Q^{(1)},\:Q^{(2)}$ of the Lyapunov equations.
\par
Finally, we present the Taylor expansion of the Lyapunov equation, \\${ Q\left(\delta\right)\:A_d\left(\delta\right)^\top+A_d\left(\delta\right)Q\left(\delta\right)+B_d\left(\delta\right)\:{B_d\left(\delta\right)}^\top=0}$ in Equation (\ref{Lyapequation1}). 
Solve successively the following three Lyapunov equations 
for the matrices $Q^{(0)},Q^{(1)},Q^{(2)}$,
\vspace{-2mm}
\begin{align}\label{lyapequation}
     0
     & = Q^{(0)}\:{J_d^{(0)}}^\top+J_d^{(0)}\:Q^{(0)}+K_d^{(0)}\:{K_d^{(0)}}^\top, \\
     0 
     & = Q^{(1)}\:{J_d^{(0)}}^\top+J_d^{(0)}\:Q^{(1)}+\left[Q^{(0)}\:{J_d^{(1)}}^\top+J_d^{(1)}\:Q^{(0)}+K_d^{(0)}\:{K_d^{(1)}}^\top+K_d^{(1)}\:{K_d^{(0)}}^\top\right],\\
     0
     & = Q^{(2)}\:{J_d^{(0)}}^\top+J_d^{(0)}\:Q^{(2)}+\left[Q^{(0)}\:{J_d^{(2)}}^\top+J_d^{(2)}\:Q^{(0)}+Q^{(1)}\:{J_d^{(1)}}^\top+J_d^{(1)}\:Q^{(1)}+\right.\\
          & \qquad \qquad \qquad \qquad \qquad \qquad\qquad \qquad~ \left.+K_d^{(0)}\:{K_d^{(2)}}^\top+K_d^{(1)}\:{K_d^{(1)}}^\top+K_d^{(2)}\:{K_d^{(0)}}^\top\right]. \nonumber
\end{align}
Since the system matrix $J_d^{(0)}$ is Hurwitz, 
each of the above three Lyapunov equations 
has a unique positive definite solution.
\subsection{The algorithm of computation of the gradient vector and the Hessian matrix of a {\em standard deviation} $\sigma_i$ with respect to a power supply vector $p_s\in P^+$}\label{gradientandHessianofS}
Subsection \ref{gradientandHessianofV} analyses 
the method of computation for the gradient vector and the Hessian matrix 
of a variance $V_i$ of a a power supply vector $p_s\in P^+$.
However, needed is the computation of a gradient vector and 
the Hessian matrix of a standard deviation $\sigma_i$
of a vector $p_s\in P^+$ denoted as~$\nabla{\sigma_i}\left(p_s\right),\:\nabla^2\sigma_i\left(p_s\right)$, respectively.  Here, we have the formulas,
 \vspace{-2mm}
\begin{align*}
  \nabla{{\sigma_i}^2}
	& = 2\:\sigma_i\:\nabla \sigma_i ~ 
	  \Rightarrow ~ 
	  \nabla{\sigma_i}
	  =\frac{\nabla{\sigma_i^2}}{2\:\sigma_i}
	  =\frac{\nabla{V_i}}{(2\:\sigma_i)}, \\
   \nabla^2\sigma_i^2
   & = 2\:\nabla\sigma_i^{\top}\:\nabla\sigma_i+2\:\sigma_i\:\nabla^2\sigma_i ~
       \Rightarrow ~
       \nabla^2\sigma_i
       = \frac{\nabla^2\sigma_i^2-2\:\nabla{\sigma_i}^{\top}\:\nabla{\sigma_i}}{                 2\:\sigma_i}
       = \frac{H\left(V_i\right)-2\:\nabla{\sigma_i}^{\top}\:\nabla{\sigma_i} }{                 2\:\sigma_i}.
\end{align*}
Define 
$\left\{G=\nabla{{V_i}},\:H=\nabla^2 V_i;G1=\nabla{{\sigma_i}},\:H1=\nabla^2\sigma_i\right\}$.
Our analysis leads to Algorithm \ref{alg:GandH}. 
%
\vspace{-3mm}
\begin{algorithm}[h]
\caption{Computation of the Gradient and the Hessian of $\sigma_i$ for a vector $p_s$.}
\label{alg:GandH}
\begin{algorithmic}
\STATE{Step 1: Specify a $p_s\in P^+$.}
\STATE{Step 2: Compute the matrices $U,\:\Lambda^{\dag}$ 
by Procedure~\ref{proc:computationcontrolobjectivefunction} Step 1;
the matrix $E$ by subsection \ref{gradientandHessianofV}; 
the matrix $W^{(0)}=W(p_s)$ by Lemma \ref{Vaindiff}.}
\FOR{$k=1,\:2,\:\cdots,\:n^+-1$}
\STATE{1. Compute the weight matrices, $W^{(1)},\:W^{(2)}$, by Equation (\ref{W1W2}).}
\STATE{2. Compute the transformation matrices, $U{\left(0\right)},\:U^{\left(1\right)},\:U^{\left(2\right)}$, by Remark (\ref{remarkonU}).}
\STATE{3. Compute the system, input and output matrices, $J_d^{(0)},\:J_d^{(1)},\:J_d^{(2)},\:K_d^{(0)},\:K_d^{(1)},\:K_d^{(2)}$,\\\qquad \qquad$C_d^{(0)},\:C_d^{(1)},\:C_d^{(2)}$, by Equations (\ref{Cd0Cd1Cd2Kd0Kd1Kd2}), (\ref{Ad0Ad1Ad2}).}
\STATE{4. Solve the Lyapunov equations of Equation (\ref{lyapequation}) to obtain $Q^{(0)},\:Q^{(1)},\:Q^{(2)}$.}
\STATE{5. Compute the component of the gradient vector and the diagonal element of the Hessian matrix, $G(k),\:H(k,k)$, by Equation (\ref{GradientandHessian}).}
\ENDFOR
\FOR{$k\:=\:1,\:2,\:\cdots,\:n^+-2$}
\FOR{$j\:=\:k+1,\:k+2,\:\cdots,\:n^+-1$}
\STATE{1. Repeat the steps 1-4 of the former loop.}
\STATE{2. Compute the off-diagonal element of the Hessian matrix by, $2\:H(k,j)\leftarrow$ Equation (\ref{GradientandHessian})$-H\left(k,k\right)-H\left(j,j\right)$.}
\ENDFOR
\ENDFOR
\STATE{Step 3: Compute $G_1=\frac{G}{2\:\sigma_i},\:H_1=\frac{H-2\:G_1^\top\:G_1}{2\:\sigma_i}$.}
\RETURN $G_1,\:H_1.$
\end{algorithmic}
\end{algorithm}
\subsection{A steepest descent method to approximate a local minimizer of the control objective function}\label{sec:steepest}
The reader will find in this section:
A description of the search for the steepest descent direction, 
primarily for Case 1.1, with comments provided for the Cases 1.2 and 2.
A description of the algorithm using the projected subgradient method 
to compute the steepest descent direction.
A description of the line search along the steepest descent direction.
A proof that the line search requires only a finite number of steps.
\begin{algorithm}[t]
\caption{A steepest descent method to approximate a local minimizer of the control objective function $f$.}
\label{alg:steep}
\begin{algorithmic}
\STATE{Step 1: Choose an initial power supply vector $p_s^{(0)}$ within the interior of the polytope $P^+$, denoted as $p_s^{(0)}\in \text{interior}(P^+)$.}
\STATE{Step 2: Compute the direction vector $v^{(i)}$ which satisfies, $f'\left(p_s^{(i)},v^{(i)}\right)=\min\limits_{b_1\leq A_1\left(p_s^{(i)}+v\right),~p_s^{(i)}+v\leq b_2}{f'\left(p_s^{(i)},v\right)},\:i\in \mathbb{N}.$  To accomplish this, one can apply Algorithm (\ref{alg:prosub}), with the formulas for matrices $A_1$, $b_1$, and $b_2$ as outlined in Equation (\ref{DODEFINITION2}).}
\STATE{Step 3: Set $t=1$ and specify two parameters $0<\alpha<0.5,\:0<\beta<1$. While $f\left(p_s^{(i)}+t\cdot v^{(i)}\right)>f\left(p_s^{(i)}\right)+\alpha\cdot t\cdot f'(p_s^{(i)},v^{(i)})$, then update $t= \beta\cdot t$.}
\STATE{Step 4: Set $p_s^{(i+1)}=p_s^{(i)}+t\cdot v^{(i)}$.}
\RETURN to Step 2.
\STATE{Determine whether to break or continue the loop based on the following cases, 
\begin{itemize}
\item Case 5.1: If $f'\left(p_s^{(i)},v^{(i)}\right)=0$ and $v^{(i)}=0$, then $p_s^{(i)}$ is a local minimizer and stop the procedure. 
\item Case 5.2: If $f'\left(p_s^{(i)},v^{(i)}\right) = 0$ and $v^{(i)} \neq 0$, then for $\xi \in (0, ~1)$ according to Def. \ref{def:xi}, compute $f''\left(p_s^{(i)} + \xi \cdot v^{(i)}, v^{(i)}\right)$. (1) If $f''\left(p_s^{(i)} + \xi \cdot v^{(i)}, v^{(i)}\right) > 0$, then $p_s^{(i)}$ is a local minimizer, and the procedure is stopped. (2) Otherwise, if $f''\left(p_s^{(i)} + \xi \cdot v^{(i)}, v^{(i)}\right) < 0$, then $p_s^{(i)}$ is an inflection point. In this case, let $p_s^{(i+1)} = p_s^{(i)} + v^{(i)}$ and return to Step 2.
\item Case 5.3: If $f'(p_s^{(i)},v^{(i)})<0$, continue with Step 3.
\end{itemize}}
\STATE{{\bf Stop criterion}: If $f\left(p_s^{(i)}\right) - f\left(p_s^{(i+1)}\right) \leq \epsilon$, where $\epsilon = 10^{-6}$, then stop the iteration, and $p_s^{(i+1)}$ can be recognized as the approximation of a local minimizer. }
\end{algorithmic}
\end{algorithm}
\begin{algorithm}[t]
\caption{A project subgradient method to compute the steepest descent direction at the iteration step $i$ of Algorithm \ref{alg:steep}.}
\label{alg:prosub}
\begin{algorithmic}
\STATE{{\bf Objective}: Compute a vector $v^*$ such that $f'\left(p_s^{(i)},v^*\right)=\min\limits_{b_1\leq A_1\left(p_s^{(i)}+v\right),\:p_s^{(i)}+v\leq b_2}{f'\left(p_s^{(i)},v\right)}$. .}
\STATE{Step 1: Define $f_d\left(v\right)=f'\left(p_s^{(i)},v\right): \mathbb{R}^{n^+-1}\to\mathbb{R}$. Set $\alpha_j=1/j^\gamma,\:0<\gamma<1$. }
\STATE{Step 2: Choose the initial vector $v_0$, which satisfies $b_1\leq A_1(p_s^{(i)}+v_0),~p_s^{(i)}+v_0\leq b_2$.}
\STATE{Step 3: Execute $v_{j+1}=P\left(p_s^{(i)}+v_{j}-\alpha_j\: g_j\right)-p_s^{(i)}$. (Here, $g_j$ represents the subgradient of $f'\left(p_s^{(i)}, v\right)$ at $v_j$, and $P: \mathbb{R}^{n^+-1} \rightarrow P^+$ is the Euclidean projection operator.}
\STATE{Step 4: Define $f_{d,min}^{j}=\min{\left\{f_{d,min}^{j-1},~f_d\left(v_{j}\right)\right\}}\Rightarrow f_{d,min}^{j}=\min{\left\{f_d\left(v_1\right), \cdots,f_d\left(v_j\right)\right\}}$.}
\STATE{{\bf Assumption 5.3.1}: This problem is a piecewise linear program with linear constraints. Due to the origin, the minimizer satisfies $f'(p_s^{(i)},v^*)\leq 0$. In the case where  $f'\left(p_s^{(i)},v^*\right)=0$, we assume that there exists at most one segment which satisfies this condition. The reason is that if there exists two or more segments which satisfy this equality, then the convex combination of these line segments will also satisfy $f'\left(p_s^{(i)},v\right)=0$, which we exclude.}
\STATE{{\bf Remark 5.3.2}: By solving a linear constraints quadratic program $\min\limits_x{\frac{1}{2}\:\|x-y\|_2^2},b_1\leq A_1\:y,y\leq b_2$ through a dynamical system that is exponentially convergent \cite{zhang2002dual}, the Euclidean projection can be programmed as an operator. Reviews of projection algorithms can be found in \cite{bauschke1996projection}. Additionally, readers can explore another method called the Charged Balls Method for projecting a vector onto a convex set in \cite[Ch. 1]{danilova2022recent}.}
\STATE{{\bf Stop criterion}: It has been established in \cite{boyd2003subgradient} that $\lim\limits_{j\to \infty}{f_{d,min}^{j+1}}=f_{d}^*\in\mathbb{R}$. If the iteration number $j$ is sufficiently large,  for instance, when $j>500$, we can stop the iteration and specify a $v_j$ such that $f_{d}\left(v_j\right)=f_{d}^{*}$. Note that $v^*=v_j$ may not be a unique solution of the preceding equation, and in accordance with {\em Assumption 5.3.1}, if $f_{d}^{*}=0$, then $v^*=0$ or $v^*$ lies on a line segment.}
\end{algorithmic}
\end{algorithm}
%
Algorithm \ref{alg:steep} 
to approximate a local minimizer of the control objective function 
is based on the steepest descent method,
\cite[Section 9.4]{boyd2004convex}.
The concept of a {\em subgradient} is defined
that is used in the computation of the steepest descent direction. 
%
\begin{definition}\label{subgradient}
Let $g: \mathbb{R}^n \rightarrow \mathbb{R}$ be a convex function.
If at $x \in \mathbb{R}^n$,
a vector $h \in \mathbb{R}^n$ satisfies that
$g\left(y\right) \geq g\left(x\right) + h^{\top} \left(y - x\right), \forall ~y \in \mathbb{R}^n$,
then the {\em subgradient of $g$ at $x$} exists, and it
is equal to the vector $h$, which
is denoted by $\partial g\left(x\right) = h$.
\end{definition}
The following definition will be used in the proof that the line search requires only a finite number of steps and the evaluation of an inflection point.
\begin{definition}\label{def:xi}
Define a function $g:\mathbb{R}_{s+}\to \mathbb{R}$ with respect to $\tau$, by $g(\tau)=f(p_s+\tau v)$. It follows from Defs. \ref{def.firstorderderivative} and \ref{directionalderivative2} that, $g(\tau)=g(0)+g'(0^+)\tau+ \frac{1}{2}g''(0^+)\tau^2+o(\tau^2)~\mbox{or},~
g(\tau)=g(0)+g'(0^+)\tau+ \frac{1}{2}g''(\xi)\tau^2,~\xi\in(0,\tau), ~
\mbox{where,}~g'(0^+)=f'(p_s,v)~\mbox{and}~g''(0^+)=f''(p_s,v).$
\end{definition}

\subsubsection*{Steepest descent in Case 1.1}
For a given vector $p_s^{(i)}$, 
where $p_s^{(i)}\in P_{(k),nz}^+$ and $p_s^{(i)}\in \interior(P^+)$, 
with $i\in \mathbb{N}$ representing the iteration step 
in Algorithm \ref{alg:steep}, 
the steepest descent direction is,
$-\nabla f\left(p_s^{(i)}\right)=-\left(\nabla f_{as,k}\left(p_s^{(i)}\right)+r\cdot\nabla\sigma_k\left(p_s^{(i)}\right)\right)$. 
In this case, there exists a unique positive value $t_i^* > 0$ such that
$p_s^{(i)} - t\cdot\nabla f\left(p_s^{(i)}\right) \notin P^+$ for $t > t_i^*$, and
$p_s^{(i)} - t\cdot\nabla f\left(p_s^{(i)}\right) \in P^+$ for $t \leq t_i^*$,
which means that the vector 
$p_s^{(i)} - t_i^*\cdot \nabla f\left(p_s^{(i)}\right)$ 
lies on one of the facets of the polytope $P^+$. 
Furthermore, the value of $t_i^*$ can be easily computed 
using the constraints of the domain $P^+$ defined in Def. \ref{dom:DODEFINITION2}.
In conclusion, 
 the steepest descent direction $v$ under constraints in this case equals: 
$-t_i^*\cdot\left(\nabla f_{as,k}\left(p_s^{(i)}\right)+r\cdot\nabla\sigma_k\left(p_s^{(i)}\right)\right)$.
\par
For a vector $p_s^{(i)}\in P_{(k),nz}^+$ and $p_s^{(i)}$ lies on the boundary of $P^+$, 
denoted as $p_s^{(i)}\in \partial P^+$, 
we encounter a scenario where the direction vector
of the control objective function at the vector $p_s^{(i)}$ 
must point into the interior of the polytope $P^+$.
For the vector $p_s^{(i)}$ which is at the corner of the polytope $P^+$,
 the direction vector of the control objective function 
is particularly challenging to express by using mathematical formulas. 
Therefore, to handle these extreme cases, 
we opt for a projected subgradient method, 
as presented in Algorithm \ref{alg:prosub}, which was 
first put forward in \cite{boyd2003subgradient}. 
Moreover, at Algorithm \ref{alg:prosub} Step 2,
the subgradient satisfies, $g_j=\nabla f_{as,k}\left(p_s^{(i)}\right)+r\cdot\nabla\sigma_k\left(p_s^{(i)}\right),\:\forall\:j\in\mathbb{N}$ in this case, 
and the more concrete expression of $g_j$ can be found in Proposition \ref{prop:firstdirectionalderivative} (a).
%
\subsubsection*{Steepest descent in the Cases 1.2 and 2}
Algorithm \ref{alg:prosub} 
will be employed to compute the steepest descent direction in Case 1.2.
Consider an iteration step $v_j$, we need to get the expression of
the subgradient of the directional derivative function $f'\left(p_s^{(i)},v\right)$ at $v_j$, denoted as $g_j$ in Step 2. 
It follows Proposition  \ref{prop:firstdirectionalderivative}, 
$f'\left(p_s^{\left(i\right)},v_j\right)=f_k'\left(p_s^{\left(i\right)},v_j\right)=|A\left(k\right)~v_j|+r\cdot\nabla\sigma_k\left(p_s^{\left(i\right)}\right) v_j$, 
the subgradient $g_j$ satisfies,
\vspace{-2mm}
\begin{equation}\label{eqn.subgradient}
g_j=\begin{cases}\left(\theta \cdot A\left(k\right)+\left(1-\theta\right)\cdot \left(-A\left(k\right)\right)+r\cdot \nabla \sigma_k\left(p_s^{\left(i\right)}\right)\right)^\top,~\forall~\theta\in [0,1], ~\mbox{if}~A\left(k\right) ~v_j=0;\\
 \left(A\left(k\right)+r\cdot \nabla \sigma_k\left(p_s^{\left(i\right)}\right)\right)^\top,\mbox{if}~A\left(k\right)~  v_j>0;
 \left(-A\left(k\right)+r \cdot\nabla \sigma_k\left(p_s^{\left(i\right)}\right)\right)^\top,\mbox{if}~A\left(k\right)~v_j<0.
 \end{cases}
 \end{equation}
\par
Consider Case 2 and an iteration step $v_j$ in Algorithm \ref{alg:prosub}. 
As indicated in Proposition \ref{prop:firstdirectionalderivative}, 
we have $f'\left(p_s^{\left(i\right)},v_j\right)=\max\limits_{k\in I_{max}\left(p_s^{\left(i\right)}\right)}{f_k'\left(p_s^{\left(i\right)},v_j\right)}$. 
To determine the subgradient $g_j$ of the function 
$f'\left(p_s^{\left(i\right)},v\right)$ at $v_j$, 
one needs to specify a $k\in I_{max}\left(p_s^{\left(i\right)}\right)$ 
such that 
$f'\left(p_s^{\left(i\right)},v_j\right) = f_k'\left(p_s^{\left(i\right)},v_j\right)$. 
Then the subgradient $g_j$ 
can be determined by Equation (\ref{eqn.subgradient}) or Case 1.1.
%
\begin{proposition}\label{linesearchfinitesteps}
Step 3 of Algorithm~\ref{alg:steep} 
can be executed in a finite number of steps.
\end{proposition}
\begin{proof}
In this proof, we use the notations $p_s,\:v$ to substitute $p_s^{\left(i\right)},\:v^{\left(i\right)}$ in order to show the generality. 
Consider the $\xi\in (0,1)$, according to Def. \ref{def:xi},
 \begin{eqnarray*}
 f''\left(p_s+\xi \cdot v,v\right)&=&\max\limits_{i\in I_{max}\left(p_s+\xi \cdot v,v\right)}{f_i''\left(p_s+\xi\cdot v,v\right)}, 
~\Rightarrow\\
f_i''\left(p_s+\xi\cdot v,v\right)&\leq& \begin{cases}\lambda_{max}\left(\left(1-\left(A\left(k\right)\cdot \left(p_s+\xi\cdot v\right)+b_k\right)^2\right)^{-\frac{3}{2}}\cdot|A\left(k\right)\:\left(p_s+\xi\cdot v\right)+b_k|\cdot   \right.\\
  \left. 
    A
    \left(
      k
    \right)^\top 
    A
    \left(
     k
    \right) 
    + r \cdot \nabla^2 \sigma_k
    \left(
      p_s + \xi \cdot v
    \right) 
  \right) \cdot \|v\|_2^2,~\\
  \qquad\qquad\qquad\qquad\qquad
  \text{if} ~ A
  \left(
    k
  \right) 
  \left(
    p_s + \xi \cdot v
  \right) + b_k \neq 0;\\
  r \cdot \lambda_{max}
  \left(
    \nabla^2 \sigma_i
    \left(
      p_s + \xi \cdot v
    \right)
  \right) \cdot \|v\|_2^2, ~ \\
  \qquad\qquad\qquad\qquad\qquad
  \text{if} ~
  A
  \left(k
  \right)~ 
  \left(
    p_s + \xi \cdot v
  \right) + b_k= 0.
\end{cases}. 
\end{eqnarray*}
Hence, $\exists~M_1>0,~\forall~p_s\in P^+, p_s+v\in P^+,\xi\in (0,1)~\mbox{such that}~ f''\left(p_s+\xi\cdot v,v\right)\leq M_1$ and 
\begin{displaymath}
 f\left(p_s+t\cdot v\right)=f\left(p_s\right)+t\cdot f'\left(p_s,v\right)+\frac{t^2}{2}\:f''\left(p_s+\xi \cdot v,v\right)\leq f\left(p_s\right)+t\cdot f'\left(p_s,v\right)+\frac{M_1}{2}t^2. 
\end{displaymath}
 By Algorithm \ref{alg:prosub}, we can obtain a direction $v^*$, such that $f'\left(p_s,v^*\right)<0$. With $\alpha\in \left(0,0.5\right)$,
\begin{align}
&f\left(p_s\right)+t\cdot f'\left(p_s,v^*\right)+\frac{M_1}{2}\:t^2\leq f\left(p_s\right)+\alpha \cdot t \cdot f'\left(p_s,v^*\right)\Rightarrow \nonumber\\
 &\frac{M_1}{2}t^2+ t\cdot f'\left(p_s,v^*\right)-\alpha \cdot t \cdot f'\left(p_s,v^*\right)\leq 0\label{inq:linesearch}.
\end{align}

This inequality can be satisfied if and only if, $0\leq t\leq \frac{2}{M_1}\cdot \left(\alpha-1\right)\cdot f'\left(p_s,v^*\right)$.
In conclusion, the loop of step 4 of Algorithm (\ref{alg:steep}) can be stopped after a finite number of iterations with a $t$ which satisfies either $t=1$~\mbox{ or}~ $t>\frac{2\beta}{M_1}\cdot \left(\alpha-1\right)\cdot f'\left(p_s,v^*\right)$.
\end{proof}
\subsection{Convergence analysis}\label{subsec:convergence}
An approximation sequence having been generated as described above, 
the natural question arises: 
Does the approximation sequence generated above converge to a minimum?
We introduce the convergence analysis theorem as follows,
\begin{theorem}\label{convergenceanalysis}
\begin{itemize}
\item[(a)] 
Algorithm 5.2 terminates after a finite number of steps.
\item[(b)] 
Algorithm (\ref{alg:steep}) generates a sequence 
$\left\{p_s^{(k)}\in P^+,\forall ~k\in \mathbb{N}\right\}$ 
such that $\lim\limits_{k\to \infty}{f\left(p_s^{\left(k\right)}\right)}=f^*$, 
where $f^*$ is a local minimum.
\end{itemize}
\end{theorem}
\begin{proof}
The claim (a) follows directly from Case 5.1 and Case 5.2 (1) of Algorithm  (\ref{alg:steep}).
 \par For the claim (b), if Algorithm \ref{alg:steep} does not break the loop according to Case 5.1 and Case 5.2 (1), then it generates a sequence of values, $\left\{f\left(p_s^{\left(0\right)}\right),f\left(p_s^{\left(1\right)}\right),\cdots|~p_s^{\left(k\right)}\in P^+,\forall~k\in \mathbb{N}\right\}$, which is strictly decreasing. From Theorem \ref{existence}, $\exists~a\in\mathbb{R}$,~such that $f\left(p_s^{\left(k\right)}\right)\geq a,~\forall~k\in \mathbb{N}$.  Therefore, this sequence can be bounded from below. By the monotone convergence theorem \cite[corollary 2.11]{browder2012mathematical}, $\lim\limits_{k\to \infty}{f\left(p_s^{\left(k\right)}\right)}=\inf_{k\in \mathbb{N}}{f\left(p_s^{\left(k\right)}\right)}:=f^*$, where $f^*$ is a local minimum. 
\end{proof}
\par
The convergence rate $c \in \left(0, ~ 1\right)$ 
defined as the smallest positive real number such that\\
$|f\left(p_s^{\left(k\right)}\right) - f\left(p_s^{\left(0\right)}\right)| \leq c^k\cdot | f\left(p_s^{\left(1\right)}\right) - f\left(p_s^{\left(0\right)}\right) |$,
\cite[p. 480]{boyd2004convex}.
The determination of the convergence rate as a function
of the parameters of our problem,
is left for a future investigation.
\subsection{A projected generalized subgradient method to compute an effective initial vector}\label{sec:computationofinitialvector}
%
The steepest descent algorithm introduced in Subsection \ref{sec:steepest} 
is only locally convergent. 
To enhance its effectiveness, 
it is essential to determine an initial vector
which allows the control objective function to decrease significantly. 
A procedure is needed 
to compute an effective initial vector for algorithm \ref{alg:steep}.
\begin{proposition}\label{prop:subgradient}
Based on the principles outlined in 
\cite[theorem 10.27]{clarke2013functional} and 
the first directional derivative provided in 
Proposition \ref{prop:firstdirectionalderivative}, 
the generalized gradient of 
$f_{k}$ for $k\in \mathbb{Z}_{n_E}$ 
at any vector $p_s\in P^+$ equals,
\begin{eqnarray*}
     \partial f_k\left(p_s\right) 
        & = &\partial f_{as,k}\left(p_s\right) +r \cdot\nabla \sigma_k\left(p_s\right)\\
        &=&\left\{
          \begin{array}{l}
           \left( 1 - \left(A\left(k\right)\:p_s + b_k\right)^2 
            \right)^{-\frac{1}{2}}\cdot
            A\left(k\right) 
            \cdot sign\left(A\left(k\right) \:p_s + b_k\right)
          + r \cdot\nabla\sigma_k\left(p_s\right),\\
          \qquad \qquad \qquad  \mbox{if} ~ A\left(k\right) \:p_s  + b_k \neq 0,\\
            \theta \cdot A\left(k\right)+\left(1-\theta\right)\cdot\left(-A\left(k\right)\right)  + r \cdot\nabla \sigma_k\left(p_s\right) ,\mbox{where}~ \theta\in \left(0,1\right), \\
              \qquad \qquad \qquad   \mbox{if} ~ A\left(k\right)\:p_s + b_k = 0.
             \end{array}
         \right.
        \end{eqnarray*}
where, $\partial f_{as,k}\left(p_s\right)$ denotes the generalized gradient of 
$f_{as,k}$ for $k\in \mathbb{Z}_{n_E}$ 
for a vector $p_s\in P^+$.
\end{proposition}
\begin{proposition}\label{def:sub1}
In accordance with \cite[corollary 4.4]{mordukhovich2013subdifferentials}, 
the generalized subgradient of the control objective function $f$ 
at a specific vector $p_s\in P^+$ 
can be formally characterized as 
$\partial f\left(p_s\right) = \convexhull \left\{ \cup\left\{ \partial f_{k}\left(p_s\right)~|~k \in I_{max}\left(p_s\right)\right\}\right\}$. 
Here, $\convexhull$ denotes the convex hull of the indicated set, and 
the set $I_{max}\left(p_s\right)$ is previously defined in 
Def.~\ref{def:cases}.
\end{proposition}
\begin{proof}
What is important at this vector is to ascertain whether, 
for every $k\in T_\epsilon \left(p_s\right)$, 
the function $f_k$ exhibits uniform subsmoothness. 
Here, the set $T_\epsilon \left(p_s\right)$ is defined as follows, 
$T_\epsilon \left(p_s\right)=\left\{k\in Z_{n_E}~|~f_k\left(p_s\right)\geq f\left(p_s\right)-\epsilon\right\}$ where, $\epsilon\geq 0$.
Consider a positive real number $\delta>0$ and 
two vectors $p_a$ and $p_b$ 
that belong to the spherical neighborhood $\mathbb{B}_\delta \left(p_s\right)$, defined as
$\mathbb{B}_\delta \left(p_s\right):=\left\{x~|~\|x-p_s\|_2\leq \delta\right\}$.
We introduce an intermediate vector 
$\xi=\theta\cdot p_a+\left(1-\theta\right)\cdot p_b$, 
where $\theta\in [0,1]$. 
The minimal eigenvalue of the Hessian matrix 
$\left(\nabla^2\sigma\left(\xi\right)\right)$ 
is denoted as $\lambda_{min}\left(\nabla^2\sigma\left(\xi\right)\right)$. 
Note that this minimal eigenvalue depends on $\delta$ and can be negative. 
\begin{align*}
     & f_k\left(p_a\right)- f_k\left(p_b\right)=f_{as,k}\left(p_a\right)+r\cdot\sigma_k\left(p_a\right)-f_{as,k}\left(p_b\right)-r\cdot\sigma_k\left(p_b\right) \\
      &\geq \partial f_{as,k}\left(p_b\right)\:\left(p_a-p_b\right)+r\cdot\nabla\sigma_k\left(p_b\right)\:\left(p_a-p_b\right)+\frac{1}{2}\cdot r\cdot \left(p_a-p_b\right)^\top\:\nabla^2\sigma_k\left(\xi\right)\:\left(p_a-p_b\right)\\
                   &\geq \partial f_{as,k}\left(p_b\right)\:\left(p_a-p_b\right)+r\cdot\nabla\sigma_k\left(p_b\right)\:\left(p_a-p_b\right)+\frac{1}{2}\cdot r\cdot\lambda_{min}\left(\nabla^2\sigma_k\left(\xi\right)\right)\cdot\|p_a-p_b\|_2^2\\
    & =\langle\partial f_k\left(p_b\right),p_a-p_b\rangle+\frac{1}{2}\cdot r \cdot\lambda_{min}\left(\nabla^2\sigma_k\left(\xi\right)\right)\cdot\|p_a-p_b\|_2^2.
\end{align*}
\end{proof}
 The functions $f_k,\:k\:\in T_\epsilon \left(p_s\right)$ then satisfy the requirment of the uniform subsmoothness defined in \cite[corollary 4.4]{mordukhovich2013subdifferentials}.
\begin{definition}\label{projectedsubgradientmethod}
Define the recursion of the projected generalized subgradient algorithm as follows,
where $P$ denotes the Euclidean projection,
\begin{align}
     p_s^{\left(k+1\right)}
     & = P\left(z^{\left(k+1\right)}\right)
       = P
	\left(
	  p_s^{ \left(k\right)}
          - \alpha_k\:\partial f\left(p_s^{\left(k\right)}\right)
        \right), \: \forall\: k\in\mathbb{N},
        \label{recursion:project} 
      ~~~ p_s^{\left(0\right)} \in P^+, ~ 
	\alpha_k\in\mathbb{R}_{s,+}. \nonumber 
\end{align}
\end{definition}
\begin{theorem}\label{convergencetheorem1}
Consider the sequence $\left\{p_s^{\left(k\right)}\in P^+,\forall~k\in \mathbb{N}\right\}$ generated by Definition \ref{projectedsubgradientmethod}, and define the minimum of the function values of the finite sequence  as,\\ $f_{min}^{\left(k\right)}=\min{\left\{f\left(p_s^{\left(1\right)}\right),\cdots,f\left(p_s^{\left(k\right)}\right)\right\}}=\min{\left\{f_{min}^{\left(k-1\right)},f\left(p^{\left(k\right)}\right)\right\}}.$
\begin{itemize}
\item [(a)] If the step length sequence satisfies $\sum_{k=1}^\infty{\alpha_k}< \infty$, for example, $\alpha_k=1/k^{1.1}$, then the sequence defined before is convergent, denoted as $\lim_{k\to \infty}{f_{min}^{\left(k\right)}}=f_a^*,\:f^*\in\mathbb{R}_+$. 
\item [(b)]If the step length sequence satisfies $\lim_{k\to \infty}{\alpha_k}=0$, for example, $\alpha_k=1/k^{0.5}$, then the sequence defined before is convergent, denoted as $\lim_{k\to \infty}{f_{min}^{\left(k\right)}}=f_b^*,\:f^*\in\mathbb{R}_+$.
\end{itemize}
\end{theorem}
\begin{proof}
(a) 
Consider the partial sum, 
$\sum_{k=n}^{m}{\left(f_{min}^{\left(k+1\right)}-f_{min}^{\left(k\right)}\right)}=f_{min}^{\left(m+1\right)}-f_{min}^{\left(n\right)}$, 
for $m>n\in\mathbb{N}$. 
Define $I(k) \in \mathbb{N}_k$ as, 
$f\left(p_s^{\left(I\left(k\right)\right)}\right)
=\min{\left\{f\left(p_s^{\left(0\right)}\right),f\left(p_s^{\left(1\right)}\right),\cdots, f\left(p_s^{\left(k\right)}\right)\right\}}$.
Consequently, 
$f_{min}^{\left(m+1\right)}-f_{min}^{\left(n\right)}=f\left(p_s^{\left(I\left(m+1\right)\right)}\right)-f\left(p_s^{\left(I\left(n\right)\right)}\right)$. 
It is obvious that $I\left(m+1\right)\geq I\left(n\right)$.
If for all $m \in \mathbb{Z}_+$ with $m>n$, $I(m+1) = I(n)$,
then we can immediately get the convergence of the sequence $\left\{f_{min}^{\left(k\right)},k\in\mathbb{N}\right\}$, because the minimum value among the 
finite sequence $\left\{f\left(p_s^{\left(0\right)}\right),\cdots, f\left(p_s^{\left(k\right)}\right)\right\}$
stay unchanged after $k>n$.
Consider the case $I\left(m+1\right)> I\left(n\right)$,
\vspace{-3mm}
\begin{align*}
     & \|f_{min}^{\left(m+1\right)}-f_{min}^{\left(n\right)}\|
       = \|f\left(p_s^{\left(I\left(m+1\right)\right)}\right)
       - f\left(p_s^{\left(I\left(n\right)\right)}\right)\|
       \leq G \cdot 
	 \|p_s^{\left(I\left(m+1\right)\right)}-p_s^{\left(I\left(n\right)\right)}\| \\
     & = G\cdot 
       \|P\left(p_s^{\left(I\left(m+1\right)-1\right)}-\alpha_{I\left(m+1\right)-1} \:\partial f\left(p_s^{\left(I\left(m+1\right)-1\right)}\right)\right)-p_s^{\left(I\left(n\right)\right)}
       \|\\
       &
       \leq G\cdot
       \|p_s^{\left(I\left(m+1\right)-1\right)}-\alpha_{I\left(m+1\right)-1} \:\partial f\left(p_s^{\left(I\left(m+1\right)-1\right)}\right)-p_s^{\left(I\left(n\right)\right)}
       \| \\
                & \leq G\cdot
       \| p_s^{\left(I\left(m+1\right)-1\right)}-p_s^{\left(I\left(n\right)\right)}
       \|
       +G \cdot \|\alpha_{I\left(m+1\right)-1} \:\partial f\left(p_s^{\left(I\left(m+1\right)-1\right)}\right)\|\\
      & \leq G \cdot 
       \|p_s^{\left(I\left(m+1\right)-1\right)}-p_s^{\left(I\left(n\right)\right)}
       \|
       + G^2\cdot \|\alpha_{I\left(m+1\right)-1} \|
	 \leq \cdots\cdots
	 \leq G^2 \cdot\sum_{i=I\left(n\right)}^{I\left(m+1\right)-1}{\|\alpha_i \|}.
\end{align*}
 If $\sum_{i=1}^\infty{\alpha_i}< \infty$,  then $\forall~ \epsilon >0, \exists ~N,~ s.t.~m>n>N,~\sum_{i=n}^m{\|\alpha_i \|}<\epsilon$.  For this $n$ which satisfies $n>N$, we can make it arbitrarily large. If for all $n>N$, $I(n)\leq N$, then it follows that
 $\lim_{n\to \infty}{f_{min}^{\left(n\right)}}= f\left(p_s^{\left(j\right)}\right),\:j\leq N$. 
 Otherwise, $I\left(n\right)>N$, we can get  $\|\sum_{k=n}^{m}{\left(f_{min}^{\left(k+1\right)}-f_{min}^{\left(k\right)}\right)}\|=\|f_{min}^{\left(m+1\right)}-f_{min}^{\left(n\right)}\|\leq G^2\cdot\sum_{i=I\left(n\right)}^{I\left(m+1\right)-1}{\|\alpha_i \|}<G^2\: \epsilon$.
 As a result, $\sum_{k=1}^{\infty}{\left(f_{min}^{\left(k+1\right)}-f_{min}^{\left(k\right)}\right)}<\infty$, leading to, $\lim_{k\to \infty}{f_{min}^k}=f_a^*<\infty$.\\
(b)  
\vspace{-5mm}
\begin{align*}
     & \| f \left(p_s^{\left(k+1\right)}\right)
	  - f\left(p_s^{\left(k\right)}\right)
       \| 
      \leq G \cdot\|p_s^{\left(k+1\right)}-p_s^{\left(k\right)}\|
       = G \cdot
	 \| P 
	    \left(
	      p_s^{\left(k\right)}
	     -\alpha_k \:\partial f\left(p_s^{\left(k\right)}\right)
	    \right)
	    - p_s^{\left(k\right)}
	  \|\\
     & \leq G \cdot
	    \| p_s^{\left(k\right)}
	       -\alpha_k \: \partial f\left(p_s^{\left(k\right)}\right)
	       - p_s^{\left(k\right)}
	    \|
       \leq G \cdot
	\|\alpha_k~ \partial f\left(p_s^{\left(k\right)}\right)\|
	\leq G^2 \cdot \|\alpha_k\|.
\end{align*}
It is evident that if $\lim_{k\to \infty}{\alpha_k}=0$, 
then $\lim_{k\to \infty}{f\left(p_s^{\left(k\right)}\right)}=f_c^*<\infty$.
We cannot make sure that the generated sequence $\{f\left(p_s^{\left(k\right)}\right)|p_s^{\left(k\right)}\in P^+,k\in\mathbb{N}\}$
is monotonic decreasing.
Consequently,  $\lim_{k\to \infty}{f_{min}^{\left(k\right)}}=f_b^*\leq f_c^*$.
\end{proof}

 \begin{remark}\label{findminimizer}
 During an actual computation of the recursion sequence, when the iteration number is sufficiently large, an approximation of $f^*$ becomes available, and a finite sequence of variables $\left\{p_s^{(1)},p_s^{(2)},\cdots,p_s^{(n)},\cdots\right\}$  is generated. Then we need to identify a vector $p_s^{(k)}$ such that $f(p_s^{(k)})=f^*$. This particular $p_s^{(k)}$ will serve as the initial vector for algorithm \ref{alg:steep}. Additionally, the step length $\alpha_k=1/k^{0.5}$ typically results in faster initial convergence compared to $\alpha_k=1/k^{1.1}$. As a consequence, it is recommended that the readers follow the following steps: (1) employ the step length of $\alpha_k=1/k^{0.5}$ to execute a finite number of iterations, yielding a vector $p_{s,initial}^*$; (2) then, with the step length $\alpha_k=1/k^{1.1}$ and $p_{s,initial}^*$ as the initial vector, proceed with the recursion anew.
 \end{remark}
 
 \newpage
 \section{Numerical Experiments}\label{sec:examples}
In this section, we provide an illustrative example to demonstrate the practical application of the algorithms introduced in this paper. The network topology is depicted in Fig. \ref{fig:jing}, with nodes labeled from $1$ to $12$. Nodes $1,2,3,4$ represent power supply nodes, while nodes $5,6,7,8,9,10,11,12$ are designated as power demand nodes. This network model emulates two adjacent rings of a power transmission system.
Details regarding the power network's parameters and the computational results are presented in Table~\ref{table:one}, and details of the iterations of the proposed algporithms are plotted in Fig. \ref{fig.four}.
\begin{figure}[h]
\centering
\begin{tikzpicture}[tbcircle/.style={circle,minimum size=13pt,inner sep=0pt},
forward tbcircle/.style={tbcircle,fill=gray!70},
backward tbcircle/.style={tbcircle,fill=black!100},
tbsquare/.style={rectangle,rounded corners,minimum width=15pt,minimum height=15pt,fill=orange!50},
tbcarrow/.style={/utils/exec=\tikzset{tbcarrow pars/.cd,#1},
line width=0.1*\pgfkeysvalueof{/tikz/tbcarrow pars/width},
draw=\pgfkeysvalueof{/tikz/tbcarrow pars/draw},
-{Triangle[line width=\pgfkeysvalueof{/tikz/tbcarrow pars/line width},
length=1.1*\pgfkeysvalueof{/tikz/tbcarrow pars/width},
width=1.1*\pgfkeysvalueof{/tikz/tbcarrow pars/width},
fill=\pgfkeysvalueof{/tikz/tbcarrow pars/fill}]},
postaction={draw=\pgfkeysvalueof{/tikz/tbcarrow pars/fill},-,
line width=\pgfkeysvalueof{/tikz/tbcarrow pars/width}-5*\pgfkeysvalueof{/tikz/tbcarrow pars/line width},
shorten <=\pgfkeysvalueof{/tikz/tbcarrow pars/line width},
shorten >=1.25*\pgfkeysvalueof{/tikz/tbcarrow pars/width}-1.5*\pgfkeysvalueof{/tikz/tbcarrow pars/line width}}
},
tbcarrow pars/.cd,
fill/.initial=white,draw/.initial=black,width/.initial=4pt,
line width/.initial=0.4pt,]


\node[forward tbcircle,draw]  at (1,1) {$\textcolor{white}{9}$};
\node[forward tbcircle,draw]  at (2,1) {$\textcolor{white}{10}$};
\node[backward tbcircle,draw]  at (3,1) {$\textcolor{white}{3}$};
\node[forward tbcircle,draw]  at (4,1) {$\textcolor{white}{11}$};
\node[forward tbcircle,draw]  at (5,1) {$\textcolor{white}{12}$};

\node[backward tbcircle,draw]  at (1,2) {$\textcolor{white}{2}$};
\node[backward tbcircle,draw]  at (5,2) {$\textcolor{white}{4}$};

\node[forward tbcircle,draw]  at (1,3) {$\textcolor{white}{5}$};
\node[forward tbcircle,draw]  at (2,3) {$\textcolor{white}{6}$};
\node[backward tbcircle,draw]  at (3,3) {$\textcolor{white}{1}$};
\node[forward tbcircle,draw]  at (4,3) {$\textcolor{white}{7}$};
\node[forward tbcircle,draw]  at (5,3) {$\textcolor{white}{8}$};


\draw[tbcarrow={fill=black!70}] (1,2-0.2) to (1,1+0.2);
\draw[tbcarrow={fill=black!70}] (1,2+0.2) to (1,3-0.2);

\draw[tbcarrow={fill=black!70}] (2-0.2,3) to (1+0.2,3);
\draw[tbcarrow={fill=black!70}] (3-0.2,3) to (2+0.2,3);
\draw[tbcarrow={fill=black!70}] (3+0.2,3) to (4-0.2,3);
\draw[tbcarrow={fill=black!70}] (4+0.2,3) to (5-0.2,3);

\draw[tbcarrow={fill=black!70}] (2-0.2,1) to (1+0.2,1);
\draw[tbcarrow={fill=black!70}] (3-0.2,1) to (2+0.2,1);
\draw[tbcarrow={fill=black!70}] (3+0.2,1) to (4-0.2,1);
\draw[tbcarrow={fill=black!70}] (4+0.2,1) to (5-0.2,1);

\draw[tbcarrow={fill=black!70}] (5,2-0.2) to (5,1+0.2);
\draw[tbcarrow={fill=black!70}] (5,2+0.2) to (5,3-0.2);
\draw[tbcarrow={fill=black!70}] (3,1+0.2) to (3,3-0.2);

\end{tikzpicture}
\caption{A connected network with two subrings}
\label{fig:jing}
\end{figure}
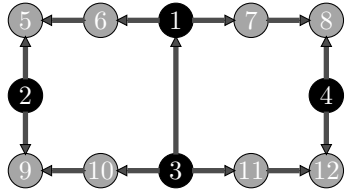

\begin{table}[h]
\begin{center}
Table 1.A. Parameter specifications for the domain of feasible vectors.
\vspace{2mm}
\begin{tabular}{|l|l|l|l|l|l|l|}
\hline
Power & $p_1^{+,max}$ & $p_2^{+,max}$ & $p_3^{+,max}$ & $p_4^{+,max}$ &-$p_5^{-}$ & -$p_6^{-}$ \\
      & 25            & 25            & 25            & 25            & -8
  & -12 \\
\hline
      &-$p_7^{-}$     & -$p_8^{-}$    &-$p_9^{-}$     & -$p_{10}^{-}$ & -$p_{11}^{-}$ & -$p_{12}^{-}$ \\

      & -13           & -7            & -8            & -12           &- 11           & -9 \\
\hline
\end{tabular}

%
          Table 1.B. The parameter specifications of the power network 
	    are listed in this table, \\
	    though not listed are 
	    that the weights of all power lines equal 30.
	\vspace{2mm}
\begin{tabular}{|l|l|l|l|l|l|l|l|l|l|l|l|l|}
 \hline
$i,\qquad~~~$node&1&2&3&4&5&6&7&8&9&10&11&12\\
 \hline
 $m_i,$\qquad~inertia&10&10&10&10&1&1&1&1&1&1&1&1\\
 \hline
$d_i,$\qquad~ damping&4&4&4&4&1&1&1&1&1&1&1&1\\
 \hline
$K_2(i,i),$~noise & 2 & 2.3 & 2.5 & 2.7 & 1.6 & 1.7 & 1.8 &1.9&1.65&1.75&1.85&2.05\\
 \hline
\end{tabular}\\
\vspace{6mm}
%
Table 1.C. The minimum and optimal power supply vector 
starting from the vector $[23,19,24]$.
\vspace{2mm}
\begin{tabular}{|l|l|l|l|l|l|l|} 
\hline
\text{Minimum}
       & $p_1^+$ & $p_2^+$ & $p_3^+$ & $p_4^+$ & -$p_5^-$ & -$p_6^-$ \\
1.0244 & 20.2247 & 17.0309 & 23.0488 & 19.6956 & -8       & -12 \\
\hline
  & -$p_7^-$ &-$p_8^-$  & -$p_9^{-}$ & -$p_{10}^{-}$ & -$p_{11}^{-}$ & -$p_{12}^{-}$\\
  & -13      & -7       & -8         & -12           & -11           & -9  \\
\hline
\end{tabular}\\
\vspace{6mm}
%
Table 1.D. The minimum and optimal power supply vector 
starting from the vector $[19,19,19]$.
\vspace{2mm}
\begin{tabular}{|l|l|l|l|l|l|l|} 
\hline
\text{Minimum}
  & $p_1^+$&$p_2^+$&$p_3^+$&$p_4^+$&-$p_5^-$&-$p_6^-$\\
\hline
1.0245 & 12.6612 & 23.2683 & 23.3900 & 20.6805 & -8 & -12 \\
\hline
  & -$p_7^-$&-$p_8^-$&-$p_9^{-}$&-$p_{10}^{-}$&-$p_{11}^{-}$&-$p_{12}^{-}$\\
\hline
  & -13 & -7 & -8 & -12 & -11 & -9 \\
\hline
\end{tabular}
\captionof{table}{Several tables.} \label{table:one}
\end{center}
\end{table}

  \par 
All numerical experiments were conducted using MATLAB 2021b on a 64-bit MacBook equipped with an M1 chip and 8GB of RAM. The combined computation time for tasks (4.a) and (4.b) was approximately 1252.137 seconds, while task (4.c) required about 310.831 seconds.
Table~\ref{table:one}.D provides results obtained using a different initial vector. A comparison between Table~\ref{table:one}.C and Table~\ref{table:one}.D reveals the presence of two distinct isolated local minimizers for the control objective function.
\begin{figure}
\begin{tabular}{ccc}
\includegraphics[width=5.3cm,height=!]{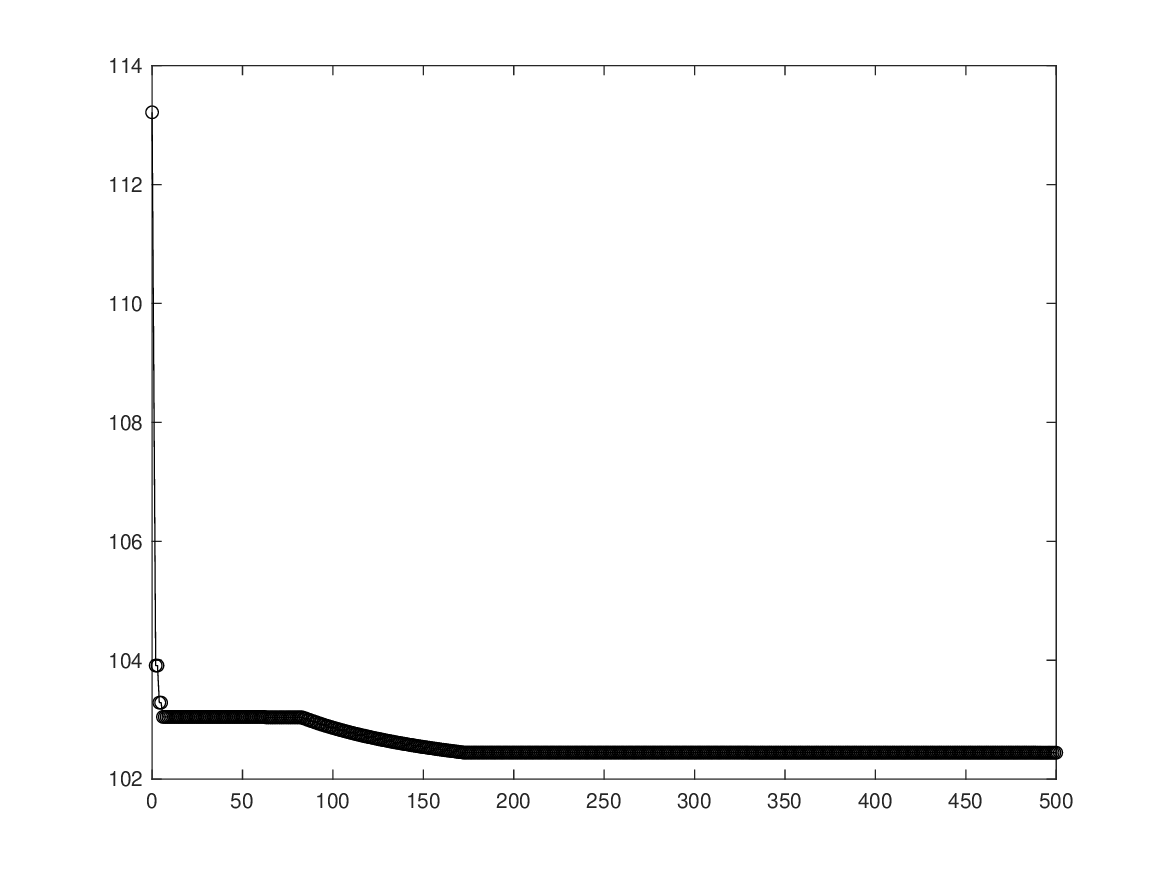} &
\includegraphics[width=5.3cm,height=!]{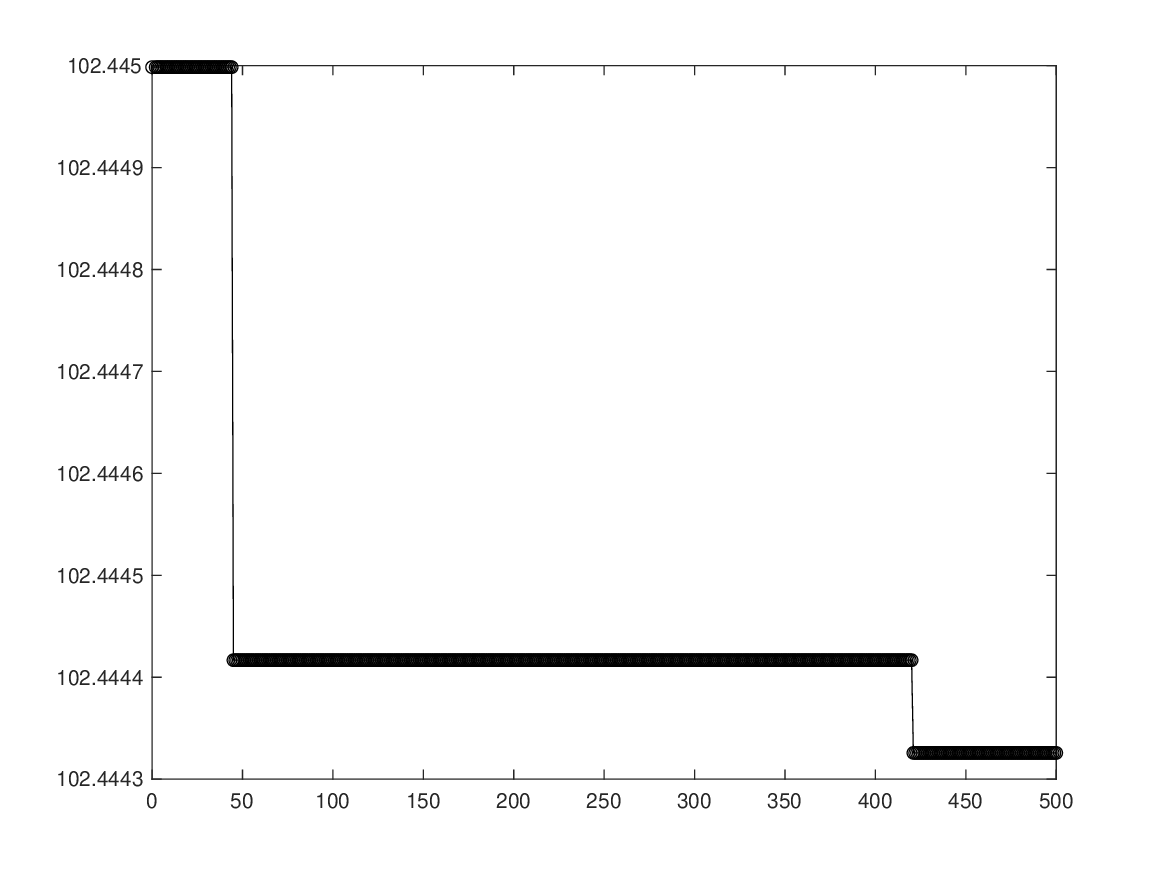} &
   \includegraphics[width=5.2cm,height=!]{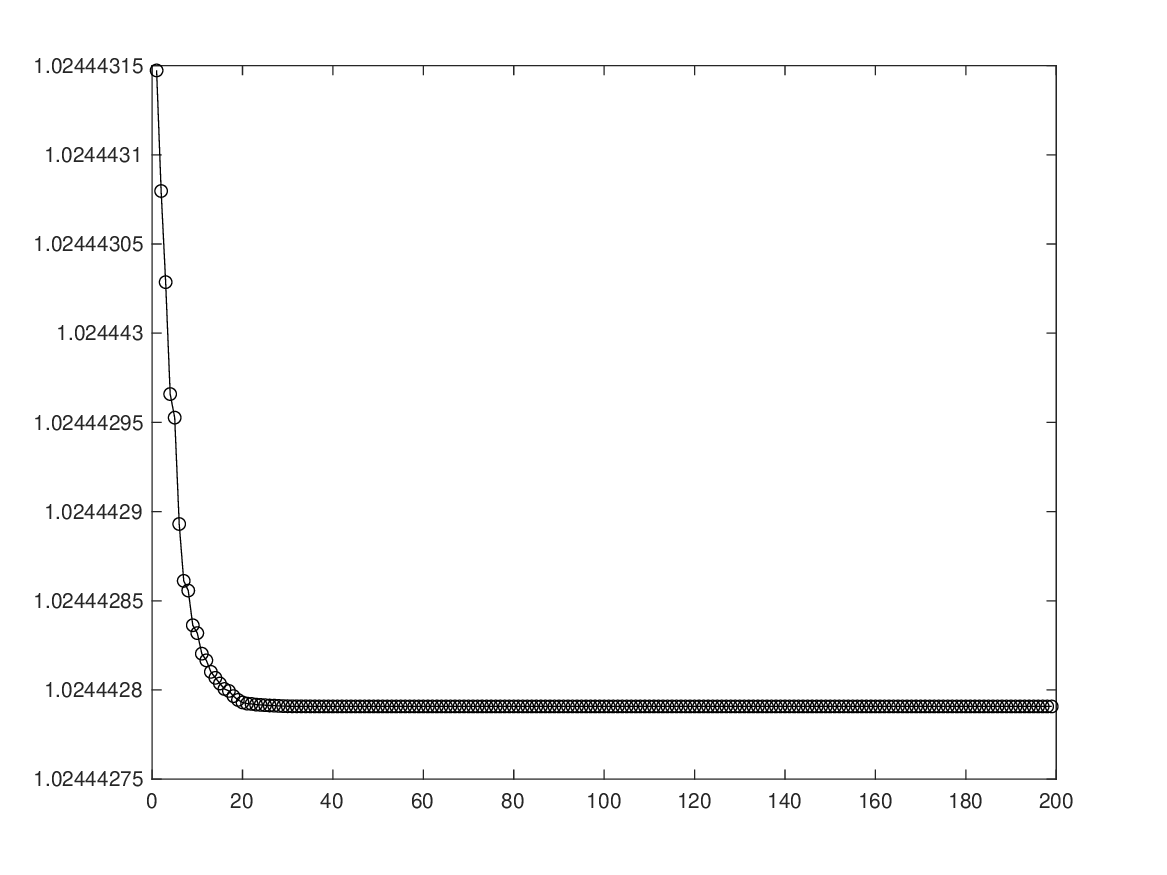}\\
 (4.a) & (4.b)& (4.c)\\
 \end{tabular}
\captionof{figure}{Figure (3.a) depicts the initial iteration of the proposed algorithm using the projected generalized subgradient method with the starting vector set as $[23,19,24]$. The x-axis represents the iteration number, and the y-axis reflects the value of the control objective function multiplied by a factor of 100.\\
Figure (3.b) shows the second iteration of the proposed algorithm using the projected generalized subgradient method. The initial vector is the optimal vector computed in (3.a). The axes are similar to (3.a).\\
Figure (3.c) displays the iteration of the proposed algorithm using the steepest descent method. The initial vector for this iteration is the optimal vector obtained in (4.b). Here, the x-axis corresponds to the iteration number, and the y-axis displays the control objective function value.}
\label{fig.four}
\end{figure}

\newpage
\section{Conclusions and further research}\label{conclusionfurtherinvestigation}
This inquiry is committed to the rigorous mathematical analysis and algorithmic resolution of optimization challenges inherent to the control of stochastic power systems. Our investigation is centered on a control objective function for ensuring the transient stability of power systems. Notably, the complexity of this optimization task is compounded by the incorporation of implicit functions, maximum, and absolute value operators within the control objective function. The algorithms proposed in this study, characterized by a polynomial complexity of $O\left(n^5\right)$, display computational feasibility but meanwhile reasonably high.
\par
In the prospective research, particular emphasis will be placed on the derivation of convergence rates and the efficiency of the advanced algorithms introduced herein. The  refinement of a second-order algorithm may represent a potential enhancement in a heightened level of precision to our methodological framework. Additionally, the exploration of algorithms with global convergence properties, as expounded in the publications such as \cite{liu2019successive} and \cite{wang2019global}, gives a possible trajectory for our future investigation.
\section{Acknowledgments}
Zhen Wang is grateful to the China Scholarship Council (No. 202106220104) for financial support for his stay at Delft University of Technology.
\newpage
\appendix
  \section{Formulas of the Matrices in Sections \ref{gradientandHessianofV} and \ref{gradientandHessianofS}}
Define $p_{sp}=\begin{bmatrix}p_s,&-p^-\end{bmatrix}$\\$\in\mathbb{R}^{n_V-1}$, the weight matrices following from the perturbation from $p_s$ to $p_s+\delta\:\mu$,
  \begin{equation}\label{W1W2}
 \begin{small}
\begin{aligned}
&\text{for}\:\:\mu=e_k,\\
W^{(1)}&=W\diag\left(-(B^{\top}U^{\top}\Lambda^{\dag}E\:p_{sp})\circ(1-(B^{\top}U^{\top}\Lambda^{\dag}E\:p_{sp})\circ^2)\circ^{-1/2}\circ(B^{\top}U^{\top}\Lambda^{\dag}E_k)\right),\\
W^{(2)}&=W\diag\left(1/2\cdot\left[-\left\{(B^{\top}U^{\top}\Lambda^{\dag}E_k)\circ(1-(B^{\top}U^{\top}\Lambda^{\dag}E\:p_{sp})\circ^2)\circ^{-1/2}\circ(B^{\top}U^{\top}\Lambda^{\dag}E_k)\right\}\right.\right.-\left\{(B^{\top}U^{\top}\Lambda^{\dag}E\:p_{sp})\circ\right.\\
&\left.\circ(-1/2\cdot(1-(B^{\top}U^{\top}\Lambda^{\dag}E\:p_{sp})\circ^2)\circ^{-3/2})\circ(-2\:B^{\top}U^{\top}\Lambda^{\dag}E\:p_{sp})\right. \left.\left.\left. \circ(B^{\top}U^{\top}\Lambda^{\dag}E_k)\circ(B^{\top}U^{\top}\Lambda^{\dag}E_k)\right\}\right]\right); \\
&\text{for}\:\:\mu=e_k+e_j,\\
W^{(1)}&=W\diag\left(-(B^{\top}U^{\top}\Lambda^{\dag}E\:p_{sp})\circ(1-(B^{\top}U^{\top}\Lambda^{\dag}E\:p_{sp})\circ^2)\circ^{-1/2}\circ(B^{\top}U^{\top}\Lambda^{\dag}(E_k+E_j)\right),\\
W^{(2)}&=W\diag\left(1/2\cdot\left[-\left\{(B^{\top}U^{\top}\Lambda^{\dag}(E_k+E_j))\circ(1-(B^{\top}U^{\top}\Lambda^{\dag}E\:p_{sp})\circ^2)\circ^{-1/2}\circ(B^{\top}U^{\top}\Lambda^{\dag}(E_k+E_j))\right\}-\right.\right.\\
& -\left\{(B^{\top}U^{\top}\Lambda^{\dag}E\:p_{sp})\circ({-1/2}\cdot(1-(B^{\top}U^{\top}\Lambda^{\dag}E\:p_{sp})\circ^2)\circ^{-3/2})\right.\circ\left.\left.\left.  (-2B^{\top}U^{\top}\Lambda^{\dag}E\:p_{sp})\circ(B^{\top}U^{\top}\Lambda^{\dag}(E_k+E_j))\circ\right.\right.\right.\\
&\circ\left.\left.\left.(B^{\top}U^{\top}\Lambda^{\dag}(E_k+E_j))\right\}\right]\right).
\end{aligned}
\end{small}
\end{equation}
Taylor expansions of a variance $V_i$ and a row vector of output matrix $C_{d,i}$ with respect to $\delta$,
\begin{equation}\label{ViCdQVi}
 \begin{small}
\begin{aligned}
V_i(p_s+\delta\mu)&=C_{d,i}(\delta)Q(\delta)C_{d,i}(\delta)^\top,Q(\delta)=Q^{(0)}+Q^{(1)}\delta+Q^{(2)}\delta^2+O(\delta^3),\\
C_{d,i}(\delta)&=C_{d,i}^{(0)}+C_{d,i}^{(1)}\delta+C_{d,i}^{(2)}\delta^2+O(\delta^3),\\
V_i(p_s+\delta\mu)&=\left(C_{d,i}^{(0)}+C_{d,i}^{(1)}\delta+C_{d,i}^{(2)}\delta^2+O(\delta^3)\right)\left(Q^{(0)}+Q^{(1)}\delta+Q^{(2)}\delta^2+O(\delta^3)\right)\left(C_{d,i}^{(0)}+C_{d,i}^{(1)}\delta+C_{d,i}^{(2)}\delta^2+O(\delta^3)\right)^\top\\
&=C_{d,i}^{(0)}Q^{(0)}{C_{d,i}^{(0)}}^\top+ \left[C_{d,i}^{(1)}Q^{(0)}{C_{d,i}^{(0)}}^\top+C_{d,i}^{(0)}Q^{(1)}{C_{d,i}^{(0)}}^\top+C_{d,i}^{(0)}Q^{(0)}{C_{d,i}^{(1)}}^\top\right]\delta+\left[C_{d,i}^{(1)}Q^{(1)}{C_{d,i}^{(0)}}^\top+\right.\\
&C_{d,i}^{(1)}Q^{(0)}{C_{d,i}^{(1)}}^\top+C_{d,i}^{(0)}Q^{(1)}{C_{d,i}^{(1)}}^\top+C_{d,i}^{(2)}Q^{(0)}{C_{d,i}^{(0)}}^\top+\left.C_{d,i}^{(0)}Q^{(2)}{C_{d,i}^{(0)}}^\top+C_{d,i}^{(0)}Q^{(0)}{C_{d,i}^{(2)}}^\top\right]\:\delta^2+O(\delta^3).
\end{aligned}
\end{small}
\end{equation}
Formulas of the first and second directional derivatives of the variance $V_i$ of a vector $p_s$,
\begin{equation}\label{GradientandHessian}
 \begin{small}
\begin{aligned}
\nabla_\mu V_i(p_s)&=\left[C_{d,i}^{(1)}Q^{(0)}{C_{d,i}^{(0)}}^\top+C_{d,i}^{(0)}Q^{(1)}{C_{d,i}^{(0)}}^\top+C_{d,i}^{(0)}Q^{(0)}{C_{d,i}^{(1)}}^\top\right],\\
\nabla^2_\mu V_i(p_s)&=2 \left[C_{d,i}^{(1)}Q^{(1)}{C_{d,i}^{(0)}}^\top+C_{d,i}^{(1)}Q^{(0)}{C_{d,i}^{(1)}}^\top+C_{d,i}^{(0)}Q^{(1)}{C_{d,i}^{(1)}}^\top+C_{d,i}^{(2)}Q^{(0)}{C_{d,i}^{(0)}}^\top+C_{d,i}^{(0)}Q^{(2)}{C_{d,i}^{(0)}}^\top+C_{d,i}^{(0)}Q^{(0)}{C_{d,i}^{(2)}}^\top\right].
\end{aligned}
 \end{small}
\end{equation}
The expressions for the output and input matrices after computing the matrices $U^{\left(0\right)}, U^{\left(1\right)}, U^{\left(2\right)}$,
\begin{equation}\label{Cd0Cd1Cd2Kd0Kd1Kd2}
 \begin{small}
\begin{aligned}
C_d^{(0)}&=\begin{bmatrix}\left(B^{\top}M^{-1/2}U^{(0)}\right)\left([2:n_{V}]\right),&0_{n_{E}\times n_{V}}\end{bmatrix},\\
C_d^{(1)}&=\begin{bmatrix}\left(B^{\top}M^{-1/2}U^{(1)}\right)\left([2:n_{V}]\right),&0_{n_{E}\times n_{V}}\end{bmatrix},\\
C_d^{(2)}&=\begin{bmatrix}\left(B^{\top}M^{-1/2}U^{(2)}\right)\left([2:n_{V}]\right),&0_{n_{E}\times n_{V}}\end{bmatrix},\\
K_d^{(0)}&=\left[\begin{array}{l}0_{(n_{V}-1)\times n_{V}}\\{U^{(0)}}^{\top}M^{-1/2}K_2\end{array}\right],\:K_d^{(1)}=\left[\begin{array}{l}0_{(n_{V}-1)\times n_{V}}\\{U^{(1)}}^{\top}M^{-1/2}K_2\end{array}\right],\:K_d^{(2)}=\left[\begin{array}{l}0_{(n_{V}-1)\times n_{V}}\\{U^{(2)}}^{\top}M^{-1/2}K_2\end{array}\right].
\end{aligned}
 \end{small}
\end{equation}
The expressions for the system matrices following the computation of the matrices $U^{\left(0\right)}, \;U^{\left(1\right)},\;U^{\left(2\right)}$,
\begin{equation}\label{Ad0Ad1Ad2}
 \begin{small}
\begin{aligned}
J_d^{(0)}&=\left[\begin{array}{ll}0_{(n_{V}-1)\times (n_{V}-1)}&0_{(n_{V}-1)\times 1}\:|\:I_{n_{V}-1}\\-\left({U^{(0)}}^{\top}M^{-1/2}BW^{(0)}M^{-1/2}U^{(0)}\right)\left([2:n_{V}]\right)&-{U^{(0)}}^{\top}M^{-1}DU^{(0)}\end{array}\right],\\
J_d^{(1)}(1,1)&=0_{(n_{V}-1)\times (n_{V}-1)},~J_d^{(1)}(1,2)=0_{(n_{V}-1)\times n_{V}},\\
J_d^{(1)}(2,1)&=-\left({U^{(1)}}^{\top}M^{-1/2}BW^{(0)}B^{\top}M^{-1/2}U^{(0)}\right)\left([2:n_{V}]\right)-\left({U^{(0)}}^{\top}M^{-1/2}BW^{(1)}B^{\top}M^{-1/2}U^{(0)}\right)([2:n_{V}])+\\
&-\left({U^{(0)}}^{\top}M^{-1/2}BW^{(0)}B^{\top}M^{-1/2}U^{(1)}\right)([2:n_{V}]),
\end{aligned}
 \end{small}
\end{equation}
\begin{equation*}
 \begin{small}
\begin{aligned}
J_d^{(1)}(2,2)&=-\left[{U^{(1)}}^{\top}M^{-1}DU^{(0)}+{U^{(0)}}^{\top}M^{-1}DU^{(1)}\right],\\
J_d^{(2)}(1,1)&=0_{(n_{V}-1)\times (n_{V}-1)},~J_d^{(2)}(1,2)=0_{(n_{V}-1)\times n_{V}},\\
J_d^{(2)}(2,1)&=-\left({U^{(1)}}^{\top}M^{-1/2}BW^{(1)}B^{\top}M^{-1/2}U^{(0)}\right)([2:n_V])-\left({U^{(1)}}^{\top}M^{-1/2}BW^{(0)}B^{\top}M^{-1/2}U^{(1)}\right)([2:n_V])+\\
&  -\left({U^{(0)}}^{\top}M^{-1/2}BW^{(1)}B^{\top}M^{-1/2}U^{(1)}\right)([2:n_V])-\left({U^{(2)}}^{\top}M^{-1/2}BW^{(0)}B^{\top}M^{-1/2}U^{(0)}\right)([2:n_V])+\\
&  -\left({U^{(0)}}^{\top}M^{-1/2}BW^{(2)}B^{\top}M^{-1/2}U^{(0)}\right)([2:n_V])-\left({U^{(0)}}^{\top}M^{-1/2}BW^{(0)}B^{\top}M^{-1/2}U^{(2)}\right)([2:n_V]),\\
J_d^{(2)}(2,2)&=-{U^{(1)}}^{\top}M^{-1}D\:U^{(1)}-{U^{(2)}}^{\top}M^{-1}DU^{(0)}-{U^{(0)}}^{\top}M^{-1}D\:U^{(2)}.
\end{aligned}
 \end{small}
\end{equation*}
Applying a Taylor expansion to the Lyapunov equation,
\begin{equation}\label{Lyapequation1}
 \begin{small}
\begin{aligned}
0&=(Q^{(0)}+Q^{(1)}\delta+Q^{(2)}\delta^2)(J_d^{(0)}+J_d^{(1)}\delta+J_d^{(2)}\delta^2)^\top+(J_d^{(0)}+J_d^{(1)}\delta+J_d^{(2)}\delta^2)(Q^{(0)}+Q^{(1)}\delta+Q^{(2)}\delta^2)+\\
&\quad ~+(K_d^{(0)}+K_d^{(1)}\delta+K_d^{(2)}\delta^2)(K_d^{(0)}+K_d^{(1)}\delta+K_d^{(2)}\delta^2)^\top\Rightarrow\\
0&=Q^{(0)}{J_d^{(0)}}^\top+J_d^{(0)}Q^{(0)}+K_d^{(0)}{K_d^{(0)}}^\top+\left[Q^{(1)}{J_d^{(0)}}^\top+J_d^{(0)}Q^{(1)}+Q^{(0)}{J_d^{(1)}}^\top+J_d^{(1)}Q^{(0)}+K_d^{(0)}{K_d^{(1)}}^\top+\right.\\
&\left. \qquad\qquad K_d^{(1)}{K_d^{(0)}}^\top\right]\delta+\left[(Q^{(2)}{J_d^{(0)}}^\top+J_d^{(0)}Q^{(2)}+Q^{(0)}{J_d^{(2)}}^\top+J_d^{(2)}Q^{(0)})+\right.\\
&\left.\qquad\qquad\qquad\qquad\qquad\qquad+Q^{(1)}{J_d^{(1)}}^\top+J_d^{(1)}Q^{(1)}+K_d^{(0)}{K_d^{(2)}}^\top+K_d^{(1)}{K_d^{(1)}}^\top+K_d^{(2)}{K_d^{(0)}}^\top\right]\delta^2\\
&\quad 
\end{aligned}
 \end{small}
\end{equation}
%
%
%
\newpage
\bibliographystyle{plain}
\bibliography{sicon}
\end{document}